\pdfoutput=1

\documentclass{article}
\usepackage{xcolor}
\definecolor{blue5}{RGB}{1,31,51}
\definecolor{green4}{RGB}{21,103,84}
\usepackage[
  colorlinks=true,          
  urlcolor=blue5,           
  linkcolor=blue5,          
  citecolor=green4,         
  ]{hyperref}
\usepackage{latexsym,amsmath,amscd,amssymb,framed,graphics}
\usepackage{amsthm}
\usepackage{thmtools, thm-restate}
\RequirePackage{enumitem}					
	\setlist[enumerate,1]{label=(\roman*), ref=(\roman*)}		
	\setlist[enumerate,2]{label=\alph*), ref=(\theenumi.\alph*)} 	
	\setlist[enumerate,3]{label=\arabic*., ref=(\theenumii.\arabic*)} 	
\usepackage{graphicx}
\usepackage{esint}
\usepackage{microtype}      
\usepackage[capitalise]{cleveref}
\usepackage[all]{xy}
\usepackage{tikz-cd} 
\usetikzlibrary{hobby}
\RequirePackage{mathtools}	

\DeclareFontFamily{OMS}{fcmsy}{\skewchar\font48 }
\DeclareFontShape{OMS}{fcmsy}{m}{n}{%
         <-5.5> [.942] cmsy5     <5.5-6.5> [.942] cmsy6
      <6.5-7.5> [.942] cmsy7     <7.5-8.5> [.942] cmsy8
      <8.5-9.5> [.942] cmsy9     <9.5->  [.942] cmsy10
      }{}
\DeclareFontShape{OMS}{fcmsy}{b}{n}{%
       <-6> [.942] cmbsy5
      <6-8> [.942] cmbsy7
      <8->  [.942] cmbsy10
      }{}
\DeclareMathAlphabet{\mathcal}{OMS}{fcmsy}{m}{n}

\renewcommand{\:}{\colon} 
\renewcommand{\tilde}{\widetilde} 
\renewcommand{\hat}{\widehat} 
\newcommand{\subeq}{\subseteq} 
\newcommand{\ssssarr}{\hbox to 15pt{\rightarrowfill}}
\newcommand{\sssarr}{\hbox to 20pt{\rightarrowfill}}
\newcommand{\ssarr}{\hbox to 30pt{\rightarrowfill}}
\newcommand{\sarr}{\hbox to 40pt{\rightarrowfill}}
\newcommand{\arr}{\hbox to 60pt{\rightarrowfill}}
\newcommand{\larr}{\hbox to 60pt{\leftarrowfill}}
\newcommand{\Arr}{\hbox to 80pt{\rightarrowfill}}

\newcommand{\ssmapright}[1]{\smash{\mathop{\ssarr}\limits^{#1}}}
\newcommand{\smapright}[1]{\smash{\mathop{\sarr}\limits^{#1}}}

\newcommand{\defeq}{
	\mathrel{ \vcenter{ 
		\baselineskip0.5ex \lineskiplimit0pt 
		\hbox{\small.}
		\hbox{\small.}
	}}
	=
}

\newcommand \al{\alpha}

\newcommand\ga{\gamma}

\renewcommand\th{\theta}

\newcommand\si{\sigma}

\newcommand\vphi{\varphi}
\newcommand\ph{\varphi}

\newcommand\om{\omega}

\newcommand\Om{\Omega}

\usepackage{xspace}
\newcommand\on{\operatorname}

\newcommand\ie{i.e.\ }
\newcommand\cf{cf.\@\xspace}
\newcommand*{\eg}{e.g.\@\xspace}

\renewcommand\o{\circ}

\newcommand\ch{\on{ch}}

\newcommand\hol{\on{hol}}

\newcommand\curv{\on{curv}}
\newcommand\Flux{\on{Flux}}
\newcommand\flux{\on{flux}}
\newcommand\ex{\on{ex}}

\newcommand\Aut{\on{Aut}}

\newcommand\cM{\mathcal{M}}
\newcommand\cP{\mathcal{P}}

\newcommand\M{\mathcal{M}}
\renewcommand\P{\mathcal{P}}

\newcommand\Diff{\on{Diff}}
\newcommand\AutGroup{\on{Aut}}

\newcommand\id{\on{id}}

\newcommand\pr{\smash[b]{\mathrm{pr}}}
\newcommand\ev{\on{ev}}

\newcommand\Ham{\on{Ham}}
\newcommand\Alt{\on{Alt}}

\newcommand\Hom{\on{Hom}}

\newcommand\ham{\on{ham}}

\newcommand\fg{\mathfrak g}

\newcommand\fh{\mathfrak h}

\newcommand\fz{\mathfrak z}
\newcommand\pa{\partial}

\newcommand\dif{{\mathop{}\!\mathrm{d}}}
\newcommand\difCE{\dif_{\mathrm{CE}}}

\newcommand\ZZ{\mathbb Z}

\newcommand\TT{\mathbb T}
\newcommand\RR{\mathbb{R}}
\newcommand\R{\mathbb{R}}
\newcommand\T{\mathbb{T}}
\newcommand\Z{\mathbb{Z}}

\newcommand\bC{\mathbb{C}}

\newcommand\X{\mathfrak X}
\newcommand\F{\mathcal F}

\newcommand\hatProduct{\mathbin{\hat{\ast}}}

\DeclarePairedDelimiter\equivClass{[}{]}				
\newcommand{\straightpath}[1]{\underline{#1\vphantom{v_0}}}
\DeclarePairedDelimiterX\scalarProd[2]{\langle}{\rangle}{#1,#2}	

\usepackage{xparse}
\NewDocumentCommand\XXXint{m m m m}{
	{{\setbox0=\hbox{$#1{#2#3}{\int_{#4}}$}\vcenter{\hbox{$#2#3$}}\kern-.5\wd0}}
}

\newcommand\intFibre{\fint}				

\ProvideDocumentCommand{\given}{}{}
\NewDocumentCommand{\setSymbol}{o}{
	\nonscript \, : 
	\allowbreak
	\nonscript \,
	\mathopen{ }
}
\DeclarePairedDelimiterX\set[1]{\{}{\}}{				
	\RenewDocumentCommand{\given}{}{\setSymbol[\delimsize]}
	#1
}

\NewDocumentCommand{\thmSep}{}{8.0pt plus 2.0pt minus 4.0pt}
\declaretheoremstyle[
		spacebelow=\thmSep,
		spaceabove=\thmSep,
		headfont=\bfseries,
		notefont=\normalfont,
		bodyfont=\normalfont\itshape,
		postheadspace=1em,
		qed=$\diamondsuit$,
		headpunct={}
]{myTheorem}

\declaretheoremstyle[
		spacebelow=\thmSep,
		spaceabove=\thmSep,
		headfont=\itshape,
		notefont=\normalfont,
		bodyfont=\normalfont,
		postheadspace=1em,
		qed=$\diamondsuit$,
		headpunct={}
]{myRemark}
\declaretheorem[numberwithin=section, style=myTheorem]{theorem}

\declaretheorem[sibling=theorem, style=myTheorem]{proposition}
\declaretheorem[sibling=theorem, style=myTheorem]{corollary}
\declaretheorem[sibling=theorem, name=Problem, style=myTheorem]{pro}
	
\declaretheorem[sibling=theorem, style=myRemark]{remark}
	\newenvironment{rem}{\begin{remark}\rm}{\end{remark}}
\declaretheorem[sibling=theorem, style=myTheorem]{example}

\begin{document}

\title{Central extensions of Lie groups
preserving a differential form
}
\author{Tobias Diez, Bas Janssens, Karl-Hermann Neeb and Cornelia Vizman}


\maketitle
\date 

\begin{abstract}
Let $M$ be a manifold with a closed, integral $(k+1)$-form $\omega$, and let
$G$ be a Fr\'echet--Lie group acting on $(M,\omega)$. 
As a generalization of
the Kostant--Souriau extension for symplectic manifolds, 
we consider a canonical class
of central extensions of $\fg$ by $\R$, indexed by
$H^{k-1}(M,\R)^*$.
We show that the image of $H_{k-1}(M,\Z)$ in $H^{k-1}(M,\R)^*$
corresponds to a lattice of Lie algebra extensions 
that integrate to smooth central extensions of $G$ by
the circle group $\T$.
The idea is to represent
a class in $H_{k-1}(M,\Z)$ by a weighted 
submanifold $(S,\beta)$, where $\beta$ is a closed, integral 
form on~$S$.
We 
use transgression of differential characters from
\( S\) and \( M \) to 
the mapping space \( C^\infty(S, M) \), and apply the Kostant--Souriau 
construction on \( C^\infty(S, M) \).

\end{abstract}

\setcounter{tocdepth}{2}
\tableofcontents 

\section{Introduction}\label{sec:intro}

In this paper, we investigate the integrability of a natural class of Lie algebra extensions 
arising from an \emph{exact} action of a (possibly infinite-dimensional) Lie group $G$ on 
an $n$-dimensional manifold $M$, endowed with a closed, integral $(k+1)$-form $\omega$.



To describe this class of extensions,
recall that the action of a Lie group $G$ with Lie algebra $\fg$ is called exact if it preserves $\omega$ and, for all $X\in \fg$, 
the insertion
of the fundamental vector field $X_M$ in $\omega$ is exact, \ie $i_{X_M}\omega = \dif\psi_{X}$
for some $(k-1)$-form $\psi_{X}$.
The fact that 
$\psi_{X}$ is not uniquely determined gives rise to a natural central extension
\begin{equation}\label{eq:ceonLAaction}
H^{k-1}(M,\R) \rightarrow \widehat{\fg} \stackrel{\pi}{\longrightarrow} \fg,
\end{equation}
which is canonically associated to the action of $\fg$ on $(M,\omega)$.
The Lie algebra $\widehat{\fg}$ is defined by 
\begin{equation}
\widehat{\fg} := \{(X,[\psi_{X}]) \in \fg \times \overline{\Omega}{}^{k-1}(M)\;:\; i_{X_M}\omega = \dif\psi_{X}\},
\end{equation}
where $[\psi_{X}]$ denotes the class of $\psi_{X}$ in $\overline{\Omega}{}^{k-1}(M) := \Omega^{k-1}(M)/\dif\Omega^{k-2}(M)$.
The Lie bracket on $\widehat{\fg}$ is given by $
[(X,[\psi_{X}]),(Y,[\phi_{Y}])] := ([X,Y], [i_{X_{M}} i_{Y_{M}} \omega])
$ (cf. \cite[Thm~13]{Ne05}).

Although we do not take the `higher category' point of view in this paper, 
let us briefly mention that 
the central extension~\eqref{eq:ceonLAaction} is part of 
an $L_{\infty}$-algebra arising naturally in multisymplectic geometry 
\cite{Rogers2012,Zambon2012}, where it provides the higher analogue of a momentum map
\cite{CalliesFregierRogersZambon2016}.
It is connected to the `quantomorphism $n$-group' of a higher prequantum bundle~\cite{FRS14},
and, for $k=2$, can be interpreted as the 
Lie 2-algebra of observables for a classical bosonic string coupled to a $B$-field~\cite{BaezHoffnungRogers2010}.


We are interested in integrability of those extensions of $\fg$ by $\R$ that factor through the 
central extension~\eqref{eq:ceonLAaction}.
Indeed, for every nonzero $\lambda \in H^{k-1}(M,\R)^*$,
the quotient $\fg_{\lambda} := \hat{\fg}/\mathrm{Ker}(\lambda)$ is a central extension 
of $\fg$ by $H^{k-1}(M,\R)/\ker(\lambda)$, which is identified with $\R$
using the functional $\lambda$.
The question, then, is to determine elements of $H^{k-1}(M,\R)^*$ for which the corresponding
Lie algebra extension $\R \rightarrow \smash{\widehat{\fg}_{\lambda}} \rightarrow \fg$ integrates to a smooth Lie group extension $\TT \rightarrow \smash{\widehat{G}{}_{\lambda}} \rightarrow G$.

It is instructive to consider the classical case where $k=1$ and $(M,\omega)$ is a symplectic manifold. 
Here the  $G$-action is exact if it is (weakly) Hamiltonian, 
and the central extension canonically associated to the Hamiltonian action is the 
\emph{Kostant--Souriau extension} (\cite[\S II.11]{Souriau}, \cite{Kostant})
\begin{equation}
H^0(M,\R) \rightarrow \widehat{\fg} \stackrel{\pi}{\longrightarrow} \fg.
\end{equation}
If $(M,\omega)$ is connected and prequantizable, then the image of $H_0(M,\Z)$ in 
$H^0(M,\R)^*$ yields a lattice of integrable Lie algebra extensions.
The corresponding Lie group extensions are obtained by pulling back
the group extension $\TT \rightarrow \mathrm{Aut}(P,\nabla) \rightarrow \mathrm{Ham}(M)$
along the Hamiltonian action $\iota \colon G \rightarrow \mathrm{Ham}(M,\omega)$, where $(P,\nabla)$
is a prequantum bundle whose curvature class is an integral multiple of $[\omega]\in H^2(M,\R)$
(cf.\ \cite[\S 2.4]{Brylinski2007}). 
Note that even if $\R \rightarrow \widehat{\fg}_{\lambda} \rightarrow \fg$ is a trivial extension,
the corresponding group extension may still be nontrivial.
In general, this will depend on the choice of prequantum line bundle $(P,\nabla)$.

A more representative example is the case where $k= n-1$ and $(M,\omega)$ is a mani\-fold
with a volume form.
Here the central extension associated to the exact divergence-free action of 
$\fg$ on $(M,\omega)$ is the 
\emph{Lichnerowicz extension}  \cite{Lichnerowicz}
\begin{equation}
H^{n-2}(M,\R) \rightarrow \widehat{\fg} \stackrel{\pi}{\longrightarrow} \fg.
\end{equation}
If $M$ is compact, then the group $G = \mathrm{Diff}_{\ex}(M,\omega)$ of exact volume-preserving 
diffeomorphisms is a Fr\'echet--Lie group. 
Its action on $(M,\omega)$ is exact by definition, and the image of $H_{n-2}(M,\Z)$ in $H^{n-2}(M,\R)^*$
yields a lattice of integrable $\R$-extensions for the Lie algebra 
$\fg = \X_{\ex}(M,\omega)$ of exact divergence-free vector fields.
The corresponding group extensions 
are the 
\emph{Ismagilov central extensions}~\cite[Sec.~25.3]{Ismagilov}.
In \cite{HV}, they were constructed for integral volume forms $\omega$ using a prequantum line bundle 
over the (infinite dimensional) nonlinear Grassmannian $\mathrm{Gr}_{n-2}(M)$ of compact, oriented submanifolds 
of codimension 2. 

The main result of this paper is
Theorem~\ref{cor:rol}, where we prove a generalization of the above results for arbitrary~$k$.
If $G$ is a connected Fr\'echet--Lie group with an exact action 
$\alpha \colon G \times M \rightarrow M$ on a compact manifold $(M,\omega)$,
and if $\omega$ is an integral $(k+1)$-form on $M$, 
then the image of $H_{k-1}(M,\Z)$ in $H^{k-1}(M,\R)^*$ yields a lattice of integrable 
Lie algebra extensions of $\fg$ by $\R$.

In fact, the requirements that $M$ be compact and that $G$ be connected
can both be relaxed. 
For classes in $H_{k-1}(M,\Z)$ that can be represented by a closed, 
oriented submanifold $S$, we show in Corollary~\ref{cor:rol2} 
that the compactness assumption on $M$ is not needed.
And rather than requiring $G$ to be connected, we require the weaker property that
every $\alpha_{g} \colon M \rightarrow M$ is homotopic to the identity.

Moreover, in this setting, the appropriate analogue of an exact action is expressed 
in terms of \emph{differential characters}.
Differential characters (or Cheeger--Simons characters) are generalizations
of principal circle bundles with connection.
In particular, the \emph{curvature} of a differential character $h$ of degree $k$ 
is a closed, integral $(k+1)$-form~$\omega$.
We give a brief overview of the theory of differential characters in \cref{s1}, and refer the reader to~\cite{CS85,BB} for further details.
Rather than requiring the action of $G$ to be exact, we require the condition 
that $\alpha^*_{g} h = h$ for all $g\in G$, where $h$ is a given differential character
with curvature~$\omega$.
In general, different characters with the same curvature may give rise to different 
group extensions with the same Lie algebra extension.


The above results rely heavily on the machinery of \emph{transgression} for 
differential characters.
Let $S$ be a compact, oriented manifold, and consider the canonical evaluation and projection maps:
\[ 
\begin{tikzcd}
M	& C^{\infty}(S,M)\times S	\arrow[swap]{l}{\ev} \arrow{d}{\pr_1} \arrow{r}{\pr_2}		&  S \\
	& C^{\infty}(S,M).										&
\end{tikzcd}
\]
On the level of differential forms, 
the \emph{hat product} \( \alpha \hatProduct \beta \) of \( \alpha \in \Omega^{\ell+1}(M) \) and \( \beta \in \Omega^{r+1}(S) \) is the differential form on the mapping space \( C^\infty(S, M) \) of degree \( 2 + \ell + r - \dim S \) obtained by pulling back \( \alpha \) and \( \beta \) to \( C^{\infty}(S,M)\times S \), and subsequently integrating over the fiber \( S \), see~\cite{Vizman2}.
If \( \alpha \) and \( \beta \) are closed, and the degrees satisfy \( \ell + r = \dim S \), then the hat product \( \alpha \hatProduct \beta \) yields a closed 2-form on \( C^\infty(S, M) \).
In this case, every smooth action of a Lie algebra \( \fg \) on \( C^\infty(S, M) \) preserving \( \alpha \hatProduct \beta \)  gives rise to a continuous \( 2 \)-cocycle \( \tau \) on \( \fg \) defined by
\begin{equation}\label{eq:tautau}
	\tau(X, Y) = (\alpha \hatProduct \beta)_\Phi (X_\Phi, Y_\Phi),
\end{equation}
where \( \Phi \in C^\infty(S, M) \), \( X, Y \in \fg \), and
 \( X_\Phi \) denotes the value of the fundamental vector field at \( \Phi \) 
 generated by the action of \( X \in \fg \) on \( C^\infty(S, M) \).
 
To construct central Lie group extensions, we refine the above setting by additional data.
Thus, we assume that \( \alpha \) and \( \beta \) are curvature forms of differential characters \( h \in \widehat H^\ell(M,\TT) \) and \( g \in \widehat H^r(S,\TT) \), respectively.
Such differential characters exist if and only if the closed forms $\alpha$ and $\beta$
are integral.
The hat product \( h \hatProduct g \) of the differential characters \( h \) and \( g \) is a differential character on \( C^\infty(S, M) \) defined in complete analogy to the hat product of differential forms (cf. equation~\eqref{eq:konijn}).
In particular, the curvature of \( h \hatProduct g \) is \( \alpha \hatProduct \beta \).
If \( \ell + r = \dim S \), then \( h \hatProduct g \) is a differential character of degree 1, 
which  
can be represented as the holonomy of 
a principal circle bundle \( \mathcal{P} \) over \( C^\infty(S, M) \), 
endowed with a connection \( \nabla \) whose curvature is \( \alpha \hatProduct \beta \) (see \cref{prop::prequantumBundleOnFunctionSpace}).
Consider a smooth action of a Fr\'echet--Lie group \( G \) on \( C^\infty(S, M) \), and assume that this action preserves \( h \hatProduct g \).
In particular, the \( G \)-action leaves \( \alpha \hatProduct \beta \) invariant.
If, furthermore, $G$ preserves the connected component $C^{\infty}(S,M)_{\Phi}$ of a map $\Phi$, 
then we obtain a central extension 
\begin{equation*}
	\label{eq:extensionsLiePrelim0}
	\T \rightarrow \widehat{G} \rightarrow G,
\end{equation*}
where $\widehat{G}$ is the group of automorphisms of $(\mathcal{P},\nabla)$
covering the $G$-action on $C^{\infty}(S,M)_{\Phi}$.
This is a \emph{smooth} central extension of Lie groups, and the
associated Lie algebra extension is characterized by the cohomology class of 
the cocycle \( \tau \) of \eqref{eq:tautau}.

In \cref{prop::centralExtensionOfDiffeoGroups}
we apply this to the case where the \( G \)-action on \( C^\infty(S, M) \) 
is induced by an action on the finite-dimensional manifolds \( S \) or \( M \), preserving all the relevant data. This immediately yields Theorem~\ref{cor:rol}
on lattices of integrable Lie algebra extensions.
The more involved example of the action of the current group \( C^\infty(S, G) \) on \( C^\infty(S, M) \) will be treated elsewhere~\cite{DJNV}.

\paragraph*{Notation:} We write $\T = \set{z \in \bC^\times \given |z| = 1}$ 
for the circle group which we identify 
with $\R/\Z$. Accordingly, we write 
$\exp_\T(t) = e^{2\pi it}$ for its exponential function. 
For a smooth manifold $M$, we call 
a curve $(\ph_t)_{0 \leq t \leq 1}$ in $\Diff(M)$ 
{\it smooth} if the map 
\[ [0,1] \times M \to M^2, \quad (t,m) \mapsto (\ph_t(m), \ph_t^{-1}(m)) \] 
 is smooth. 
We write $\Diff(M)_0 \subeq\Diff(M)$ for the normal 
subgroup of those diffeomorphisms $\ph$ for which there exists a smooth 
curve from $\id_M$ to $\ph$. 
All finite dimensional manifolds $M$ are required to be second countable.

\paragraph*{Acknowledgments:} K.-H.~Neeb and C.~Vizman thank the 
Erwin-Schr\"o\-dinger-Institute for Mathematical Physics for 
hospitality and a stimulating working atmosphere during a research stay 
within the program ``Infinite-dimensional Riemannian geometry with applications 
to image matching and shape analysis''. 
C.~Vizman was supported by a grant of Ministery of Research and Innovation, CNCS - UEFISCDI, 
project number PN-III-P4-ID-PCE-2016-0778, within PNCDI III.
B.~Janssens was supported by the 
NWO grant 613.001.214 ``Generalised Lie algebra sheaves", and by the NWO grant 639.032.734
``Cohomology and representation theory of infinite dimensional Lie groups''.
T.~Diez gratefully acknowledges support by the Max Planck Institute for Mathematics in the Sciences (Leipzig),
and by the above NWO grant 639.032.734.

\section{Differential characters and line bundles}
\label{s1}

We introduce differential characters following~\cite{CS85} and~\cite{BB}.
In this section $M$ denotes a locally convex smooth manifold 
for which the de Rham isomorphism holds\footnote{See~\cite[Thm.~34.7]{KrMi97} 
for a de Rham Theorem in this context and sufficient criteria for it to hold.}.
Let $C_k(M)$ be the group of smooth singular $k$-chains, and let $Z_k(M)$ and $B_k(M)$ denote the groups of $k$-cycles and $k$-boundaries,
so that $H_k(M) \defeq Z_k(M)/B_k(M)$ is the $k$-th smooth singular homology group.

A \emph{differential character (Cheeger--Simons character) of degree $k$} 
is a group homomorphism 
$h:Z_k(M)\to\TT$ for which there exists a 
differential form $\om\in\Om^{k+1}(M)$ such that 
$h(\pa\si)= \exp_\T\bigl(\int_\si\om\bigr)$ holds for every $\sigma \in C_{k+1}(M)$. 
Then $\omega$ is uniquely determined by $h$.
It is called the \emph{curvature of $h$} and will be denoted by $\curv(h)$.
We write 
\[ \widehat H^{k}(M,\TT) \subeq \Hom(Z_{k}(M),\T) \]
for the group of differential characters of degree $k$.\begin{footnote}
{In~\cite{BB} this group is denoted $\hat H^{k+1}(M,\Z)$. In this sense our notations 
are compatible, although the degree is shifted by $1$. Our convention 
follows the original one introduced by Cheeger and Simons in~\cite{CS85}.}  
\end{footnote}
It contains the cohomology group $H^{k}(M,\TT)\cong 
\Hom(H_{k}(M),\TT)\subeq \Hom(Z_k(M),\T)$ 
as the subgroup of differential characters that vanish on $k$-boundaries. 
This is the subgroup of differential characters with curvature zero.
For $\omega= \curv(h)$, the relation \( 1 = h(\pa \varsigma) = \exp_\T\bigl(\int_\varsigma\om\bigr) \) for every \( \varsigma \in Z_{k+1}(M) \)
implies that $\omega$ belongs to the abelian group 
\[ \Omega_\ZZ^{k+1}(M) \defeq \set*{\omega \in \Omega^{k+1}(M) \given \int_{Z_{k+1}(M)} 
\omega \subeq \Z}\]
of forms with integral periods.
Note that such forms are automatically closed. 
We thus get an exact sequence 
\begin{equation}
  \label{eq:1}
0\to H^{k}(M,\TT)\hookrightarrow \widehat H^{k}(M,\TT)\smapright{\curv}\Om_\ZZ^{k+1}(M)\to 0.
\end{equation}

A differential character $h$ of degree $k$ also
determines a \emph{characteristic class} $\ch(h)\in H^{k+1}(M,\ZZ)$.
In order to define $\ch(h)$, let
$\tilde h \: Z_{k}(M) \to \R$ be a homomorphism that lifts $h$ 
(such a homomorphism exists because $Z_{k}(M)$ is a free abelian 
group). Consider the cocycle $\hat h \: C_{k+1}(M) \to \Z$ given by
\[\hat h(c) \defeq \int_c \curv(h)  - \tilde h(\partial c)\]
and define \( \ch(h) \) to be the class of \( \hat{h} \) in \( H^{k+1}(M,\Z) \).
The image of $\ch(h)$ in $H^{k+1}(M,\RR)$ coincides with the class of $\curv(h)$,
when we identify $\RR$-valued cohomology and de Rham cohomology
via the de Rham isomorphism.
In particular, $\curv(h)$ is exact if and only if the image of $\ch(h)$ in \( H^{k+1}(M, \R) \) is trivial. 
To each differential form $\al\in\Om^k(M)$ one assigns a differential character 
$i(\al)$ defined by
\begin{equation}
	\label{eq:trivialBundleInclusion}
	i(\al)(c) \defeq \exp_\T\bigg(\int_c\al\bigg)
\end{equation}
for \( c\in Z_k(M) \).
Note that 
\begin{equation}
  \label{eq:theta}
\curv(i(\alpha)) = \dif \alpha.
\end{equation}
The kernel of $i$ is $\Om^k_{\ZZ}(M)$.
We thus obtain an exact sequence, \cf~\cite[Eq.~30]{BB},
\begin{equation}
  \label{eq:ch-seq}
0\to \Omega^k_\Z(M) \to \Om^k(M)\ssmapright{i} \widehat H^{k}(M,\TT)\ssmapright{\ch} 
H^{k+1}(M,\ZZ)\to 0.
\end{equation}

Differential characters of degree 1 can be realized as the holonomy of a 
principal circle bundle $P\to M$ with connection form $\th\in\Om^1(P)$.
Indeed, the holonomy map $h$ assigns to each
piecewise smooth $1$-cycle $c\in Z_1(M)$ an element $h(c)\in\TT$.
If the connection has curvature $\om\in\Om^2(M)$, then 
$h(\partial \sigma) = \exp_\T(\int_{\sigma}\om)$ for all \( \sigma \in Z_2(M) \), so that  
$h\in \widehat H^1(M,\TT)$ is a differential character whose curvature
is the same as that of the connection, $\curv(h) = \omega$.
The class $\ch(h)\in H^2(M,\ZZ)$, which does not depend on $\theta$,
is the \emph{first Chern class} of the bundle $P$.  
In this paper, we will often need the following well-known result. 

\begin{restatable}{proposition}{HOneIsBundles}\label{thm:H1isbundles}
The map $(P,\theta) \mapsto h$ that assigns to a principal circle bundle \( P \) with connection \( \theta \) its holonomy map \( h \)  
defines an isomorphism from
the group 
of equivalence classes of pairs $(P,\theta)$ modulo bundle automorphisms onto 
the group $\widehat H^1(M,\TT)$ of differential characters of degree \( 1 \). 
\end{restatable}

This can be derived from the fact that principal circle bundles with connection are 
classified by Deligne cohomology~\cite[Thm.~2.2.12]{Brylinski2007}, which is an alternative 
model for differential cohomology, \cf~\cite[Sec.~5.2]{BB}.
We give a direct proof in \cref{sec::diffCharactersDegreeOne}.


\section{Stabilizers of a differential character}
\label{s2}

For any manifold \( M \), the right action of the diffeomorphism group  \( \Diff(M) \) on \( \widehat H^k(M, \TT) \) by pull-back
\begin{equation}
	\varphi^* h (c) \defeq h(\varphi \circ c), \quad \text{ for } \varphi \in \Diff(M), 
	\quad h \in \widehat H^k(M, \TT)
\end{equation}
extends to the exact sequence~\eqref{eq:1} of abelian groups:
\begin{equation*}
\xymatrix{
0\ar[r]&H^k(M,\TT)\ar[d]^{\vphi^*}\ar[r] 
&\widehat H^{k}(M,\TT) \ar[d]^{\vphi^*}\ar[r] ^{\curv}
& \Omega_{\Z}^{k+1}(M)\ar[d]^{\vphi^*} \ar[r] & 0\\
0\ar[r]& H^k(M,\TT)\ar[r] &\widehat H^{k}(M,\TT) \ar[r]^{\curv} 
& \Omega_{\Z}^{k+1}(M) \ar[r] & 0.
}
\end{equation*}
Let \( \Diff(M)_0 \) be the subgroup of \( \Diff(M) \) consisting of those diffeomorphisms $\ph$ for which there exists a smooth curve from $\id_M$ to $\ph$.
Note that the action of \( \Diff(M)_0 \) on $H^k(M,\TT)$ is trivial: as 
$\ph \circ c - c$ is a boundary for every $c\in Z_{k}(M)$, we find 
$(\ph^*a - a)(c) = a(\ph \circ c - c) = 1$ for all
$a \in H^k(M,\TT)$. 

Denote the stabilizer group of $\omega \in \Omega^{k+1}_{\Z}(M)$ by 
\begin{equation}
  \label{eq:diff-m-omega}
\Diff(M,\om) \defeq \set{\ph \in \Diff(M) \given \ph^* \om = \om },
\end{equation}
and the stabilizer Lie algebra by 
\begin{equation}
  \label{eq:x-m-omega}
\X(M,\om) \defeq \set{X \in \X(M) \given L_{X}\om = 0}.
\end{equation}

\begin{rem}
Of particular relevance is the case where $\om \in \Omega^{k+1}(M)$ is nondegenerate, 
in the sense that $\om^{\flat} \colon T_mM \to {\rm Alt}^k(T_{m}M,\R), v 
\mapsto i_v \omega_m$
is injective for all $m\in M$. For $k=1$ and $k=\mathrm{dim}(M)-1$, this leads to 
the groups of \emph{symplectic} and \emph{volume-preserving} 
transformations, respectively. If $M$ is compact, then the symplectic and 
volume-preserving diffeomorphism groups $\Diff(M,\omega)$ are Fr\'echet Lie groups,
with Lie algebra $\X(M, \omega)$ consisting of \emph{symplectic} or \emph{divergence free} 
vector fields~\cite[Thm. 43.7, 43.12]{KrMi97}.
\end{rem}

\subsection{Stabilizer groups}


For a differential character 
$h \in \widehat H^{k}(M,\TT)$, we denote its stabilizer by
\begin{equation}
  \label{eq:diff-m-h}
\Diff(M,h) \defeq \set{\ph \in \Diff(M)\given \ph^* h = h}.
\end{equation}

\begin{rem}\label{holo}
Let $h\in \hat H^1(M,\TT)$ be the differential character defined by the holonomy
of the principal $\TT$-bundle $(P,\th)\to(M,\om)$ with curvature $\omega$. Then 
$\ph \in \Diff(M)$ satisfies $\ph^*h = h$ if and only if, for every smooth 
loop $\gamma$ in $M$, the holonomy of the loop $\gamma$ coincides with the 
holonomy of~$\ph\o\ga$. Since this is equivalent 
to the existence of a connection-preserving lift $\tilde\ph \in \Aut(P,\theta)$ 
by~\cite[Thm.~2.7]{NV}, one can view $\Diff(M,h)$ as the group of 
\emph{liftable} diffeomorphisms if $k=1$. 
\end{rem}

Consider the left action \( (\varphi, h) \mapsto \varphi^{-1 *}h \) of $\Diff(M,\om)$ on $\hat H^{k}(M,\T)$.
By~\eqref{eq:1}, the preimage \( \curv^{-1}(\omega) \) is a 
principal homogeneous space (a torsor) for the 
abelian group $H^k(M,\TT)$. 
For every $h$ with $\curv(h) = \omega$, we thus get a \( 1 \)-cocycle 
\begin{equation}\label{fluxh}
\Flux_{h} \colon \Diff(M,\om) \rightarrow H^{k}(M,\TT), \quad
\Flux_{h}(\ph) = (\ph^{-1})^*h - h \,.
\end{equation}
Note that the cohomology class of this cocycle 
$\Flux_h$ in $H^1(\Diff(M,\om), H^k(M,\TT))$ 
is independent of the choice of~$h$. 
The kernel of $\Flux_h$ is the subgroup $\Diff(M,h)$, so that we obtain an exact sequence
\begin{equation}\label{eq:rijtjemeth}
1 \rightarrow \Diff(M,h) \rightarrow \Diff(M,\om) \smapright{\Flux_{h}} H^k(M,\TT)\,,
\end{equation}
where $\Flux_h$ is only a cocycle and not necessarily a group homomorphism.
Since $\Diff(M)_{0}$ acts trivially on $H^{k}(M,\TT)$, 
the restriction of $\Flux_{h}$ to the intersection $\Diff(M,\om)\cap \Diff(M)_{0}$
is a group homomorphism. 
As this homomorphism does not depend on the choice of $h$ with $\curv(h) = \omega$, 
we denote the homomorphism by \( \Flux_\omega \),  and its kernel by
\[ \overline{\Diff}_{\ex}(M,\om) \defeq \Diff(M,h)\cap \Diff(M)_{0}.\] 
Since $\Diff(M)_{0}$ leaves the 
characteristic class of $h$ invariant, the image of $\Flux_\omega$ is contained in the 
kernel of $\ch|_{H^k(M,\T)}$, which can be identified with 
the Jacobian torus 
\[ J^k(M) 
\defeq 
{H^k(M,\R)/H^k(M,\R)_{\Z}\cong  \Hom(H_k(M),\R)/\Hom(H_k(M),\Z) }
\]
the quotient of $H^k(M,\R)$ by the image of $H^k(M,\Z)$.
Summarizing the above discussion, we obtain:
\begin{proposition}
	For any manifold \( M \) and \( \omega \in \Omega_{\Z}^{k+1}(M) \), we have an exact sequence \emph{of groups}
	\begin{equation}\label{eq:rijtjemetnul}
		1 \rightarrow
		\overline{\Diff}_{\ex}(M,\om) \longrightarrow
		\Diff(M,\om)\cap \Diff(M)_{0} \smapright{\Flux_\omega}
		J^k(M)\,.
		\qedhere
	\end{equation}
\end{proposition}

\begin{remark}
To define the group 
$\Diff_{\ex}(M,\om)$ of \emph{exact $\omega$-preserving diffeomorphisms}, note that
by a straightforward calculation,
the restriction of \( \Flux_\omega \) to \( \Diff(M,\om)_0 \)
satisfies
\[ \Flux_{\omega}(\ph) = \bigg[ \int_0^1 i_{\delta^l\ph_t} \omega \, \dif t\bigg]\,,\] 
where $(\ph_t)_{0 \leq t\leq 1}$ is a smooth path in $\Diff(M,\omega)$ 
from $\ph_0 = \id_{M}$ to $\ph_1 = \ph$. 
This shows that the restriction of $\Flux_{\omega}$ to
$\Diff(M,\omega)_{0}$ recovers
the flux homomorphism defined in~\cite{Calabi1970,Banyaga1978}, \cf also~\cite[Prop.~1.8]{NV}.
If we define the exact $\omega$-preserving diffeomorphism group by
\begin{equation}
	\Diff_{\ex}(M,\om) \defeq \overline{\Diff}_{\ex}(M,\om)\cap \Diff(M,\omega)_{0} = \Diff(M, h) \cap \Diff(M,\omega)_{0},
\end{equation}
then we obtain the exact sequence of groups
\begin{equation}
	\label{eq:diffExactFluxSequence}
	1 \rightarrow \Diff_{\ex}(M,\om)
	\longrightarrow \Diff(M,\om)_0 
	\smapright{\Flux_{\omega}} J^k(M)\,.
\end{equation}
The elements of $\Diff_{\ex}(M,\om)$ are called exact $\om$-preserving diffeomorphisms. 
In general, the group $\overline{\Diff}_{\ex}(M,\om)$ may be strictly larger 
than the group 
$\Diff_{\ex}(M,\om)$ of exact $\omega$-preserving diffeomorphisms.
\end{remark}

\subsection{Stabilizer Lie algebras}
We now continue with the stabilizers at the infinitesimal level.  
The Lie algebra homomorphism
\begin{equation}\label{infl}
\flux_\omega:\X(M,\om)\to H^k(M,\R),\quad \flux_\omega(X)=[i_X\om]
\end{equation}
is the {\it infinitesimal flux} cocycle whose kernel is the ideal 
\[
\X_{\ex}(M,\om) \defeq \set{X \in \X(M,\omega)\given i_{X} \om \text{ is exact}}.
\]
If $h\in \widehat H^{k}(M,\TT)$ has curvature $\omega$, then
the groups $\Diff(M,h)$, $\overline{\Diff}_{\ex}(M,\om)$ and 
$\Diff_{\ex}(M,\om)$
have $\X_{\ex}(M,\om)$ as their Lie algebra in the sense that, for 
every smooth curve $(\ph_t)_{t \in [0,1]}$ in $\Diff(M)$ 
with $\ph_0 =\id_M$, the curve $(\ph_t)$ is contained in the group 
if and only if its logarithmic derivative 
\begin{equation}
  \label{eq:logder}
\delta^l\ph_t \defeq T(\ph_t)^{-1} {\textstyle \frac{d}{dt}} \ph_t \in \X(M)
\end{equation}
takes values in $\X_{\ex}(M,\om)$.
Accordingly, the exact sequence of Lie algebras associated to~\eqref{eq:rijtjemeth},
\eqref{eq:rijtjemetnul} and~\eqref{eq:diffExactFluxSequence} is 
\begin{equation}\label{infi}
0 \rightarrow \X_{\ex}(M,\om) \longrightarrow \X(M,\om) \smapright{\flux_\omega} 
H^{k}(M,\R)\,.
\end{equation}

\begin{proposition} \label{prop:threegroups}
Let $h\in \widehat H^{k}(M,\TT)$ be a differential character with curvature $\omega$.
All three groups $\Diff(M,h)$,
$\overline{\Diff}_{\ex}(M,\om)$, and $\Diff_{\ex}(M,\om)$
have the same smooth arc-component and the same Lie algebra 
$\X_{\rm ex}(M,\omega)$.  
\end{proposition}
 
\begin{remark} 
For manifolds $M$ with trivial $k$-th homology, the relations between the above 
groups simplify.
If $H_k(M) = \set{0}$, then \( \Diff(M,h) = \Diff(M,\om) \), as the last term in 
\eqref{eq:rijtjemeth} vanishes. 
If $H_k(M)$ is a torsion group, then $J^k(M) = 0$, 
so \( \overline{\Diff}_{\ex}(M,\om) = \Diff(M,\om)\cap \Diff(M)_{0}\)
by \eqref{eq:rijtjemetnul},
and \(\Diff_{\ex}(M,\om) = \Diff(M,\om)_{0}\)
by \eqref{eq:diffExactFluxSequence}.
\end{remark}

\begin{remark} If $\om$ is nondegenerate and $k=1$ or $k=n-1$, then the infinitesimal flux homomorphism \( \flux_\omega \)
is surjective.
\begin{enumerate}
\item For $k=1$, the form $\om$ is symplectic, $\X_{\ex}(M,\om)$ is the Lie algebra 
$\X_{\ham}(M,\om)$ of \emph{Hamiltonian vector fields} and  the group 
$\Ham(M,\om)$ of \emph{Hamiltonian diffeomorphisms} coincides with the identity component 
$\Diff_{\ex}(M,\om)_0$.

\item For $k=n-1$,  $\om$ is a volume form, so we get the Lie algebra of \emph{exact divergence free vector fields} $\X_{\ex}(M,\om)$, 
and the group of \emph{exact volume-preserving diffeomorphisms} $\Diff_{\ex}(M,\om)_0$.
\qedhere
\color{black}
\end{enumerate}
\end{remark} 

\subsection{Fr\'echet--Lie groups}

If $M$ is a compact manifold with symplectic form $\omega$, then 
the group $\Diff(M,\om)$ of symplectomorphisms is a 
Fr\'echet--Lie group acting smoothly on $M$.
If $\mu \in \Omega^{d}(M)$ is a volume form, then the same holds for 
the group $\Diff(M,\mu)$ of volume-preserving diffeomorphisms
\cite[\S43]{KrMi97}.
From the following proposition, we see in particular that
$\Ham(M,\omega) \subseteq \Diff(M,\omega)_{0}$ and $\Diff_{\ex}(M,\mu) \subseteq \Diff(M,\mu)_{0}$ are embedded Fr\'echet--Lie subgroups.

\begin{proposition}\label{LieGroups}
Let $M$ be a compact manifold, and let $h\in \widehat H^{k}(M,\TT)$ 
be a differential character with curvature 
$\omega \in \Omega^{k+1}_\Z(M)$. 
Let $G$ be a Fr\'echet--Lie group
with Lie algebra $\fg$, and let 
$i \colon G \rightarrow \Diff(M,\om)$ be a group homomorphism
that induces a smooth action $G\times M \rightarrow M$ 
and hence a Lie algebra homomorphism $i_* \colon \fg \rightarrow \X(M,\om)$.
Then $G_{\ex} \defeq \set{g\in G \given \Flux_{h}(i(g)) = 0}$
is an embedded Lie subgroup of $G$ with Lie algebra
$\fg_{\ex} \defeq \set{\xi \in \fg \given [i_{i_*(\xi)}\omega]=0}$.
\end{proposition}
\begin{proof}
Since $G$ acts smoothly on $M$, the homomorphism 
$i \colon G \rightarrow \Diff(M,\omega)$ is a morphism of 
diffeological spaces.
Since the same holds for the flux cocycle
$\Flux_{h}\colon \Diff(M,\omega) \rightarrow H^k(M,\T)$,
the cocycle ${\Flux_{h}} \circ i \colon G \rightarrow H^k(M,\T)$
is a morphism of diffeological spaces as well. 
Since $M$ is compact, $H^k(M,\T)$ is a finite dimensional Lie group.
In particular, $G$ and $H^k(M,\T)$ are both objects in the category 
of Fr\'echet manifolds, which constitutes a full subcategory of 
the category of diffeological spaces~\cite{Losik1994}. 
It follows that ${\Flux_{h}} \circ i \colon G \rightarrow H^k(M,\T)$ is smooth, 
with derivative ${\flux_\omega} \circ i_* \colon \fg \rightarrow H^k(M,\R)$ at the identity.
Since $H^{k}(M,\R)$ is finite dimensional, we conclude from
Gl\"ockner's Implicit Function Theorem~\cite{Gloeckner}
(or more precisely, from the Regular Value Theorem~\cite[Thm.~III.11]{NW08} derived from this),
that the kernel $G_{\ex}$ of ${\Flux_{h}} \circ i$ is a split embedded 
submanifold of $G$. It is a Lie subgroup, because ${\Flux_{h}} \circ i$
is a cocycle. The corresponding Lie algebra $\fg_{\ex}$ is the kernel of  
${\flux_\omega} \circ i_*$.
\end{proof}

\section{Transgression of differential characters}

Let $S$ be a compact, oriented manifold, and let $M$ be a finite dimensional manifold.
We now describe a natural way to combine differential characters from \( S \) and from \( M \) to a differential character on the 
mapping space \( C^\infty(S, M) \), which carries a natural 
Fr\'echet manifold structure~\cite{KrMi97}.
For this, we consider the canonical evaluation and projection maps:
\[ 
\begin{tikzcd}
M	& C^{\infty}(S,M)\times S	\arrow[swap]{l}{\ev} \arrow{d}{\pr_1} \arrow{r}{\pr_2}		&  S \\
	& C^{\infty}(S,M).										&
\end{tikzcd}
\]
We define the hat product \( h \hatProduct g \) of \( h \in \widehat H^\ell(M,\TT) \) and \( g \in \widehat H^r(S,\TT) \) as the 
differential character on \( C^\infty(S, M) \) of degree \( 1 +\ell + r - \dim S \) obtained by pulling back \( h \) and \( g \) to \( C^{\infty}(S,M)\times S \)
and integrating over the fiber \( S \):
\begin{equation}\label{eq:konijn}
	h \hatProduct g = \intFibre_S \ev^* h \ast \pr_2^*\, g.
\end{equation}
Here \( \ast \) denotes the product \( \widehat H^\ell(Q, \TT) \times \widehat H^r(Q, \TT) \to \widehat H^{\ell+r +1}(Q, \TT) \) of differential characters, see~\cite[Def.~28]{BB}, and \( \intFibre_S\: \widehat H^\star(Q \times S, \TT) \to \widehat H^{\star-\dim S}(Q, \TT) \) denotes fiber integration, see~\cite[Def.~38]{BB}.
Since the curvature map intertwines these operations with the wedge product and with fiber integration of differential forms, respectively, we obtain
\begin{equation}\label{commuteswithoperations}
	\curv (h \hatProduct g) = \intFibre_S \ev^* \curv (h) \wedge \pr_2^* \curv (g).
\end{equation}
This expression coincides with the hat product \( \curv h \hatProduct \curv g \) of ordinary differential forms as introduced in~\cite{Vizman2}.

Note that $\Diff(M)$ acts on $C^{\infty}(S,M)$ by $A(\phi)(f) = \phi \circ f$, and that  
$\Diff(S)$ acts by $B(\psi)(f) = f \circ \psi^{-1}$.
Let \( \Diff(S)_+ \) be the subgroup of \( \Diff(S) \) that consists of orientation-preserving diffeomorphisms.

\begin{proposition}\label{prop:transgress}
The hat product is equivariant with respect to $\Diff(M)$ and $\Diff(S)_+$, that is,
$A(\phi)^* (h \hatProduct g) = (\phi^*h) \hatProduct g$ and $B(\psi)^* (h \hatProduct g) = h \hatProduct (\psi^* g)$.
\end{proposition}
\begin{proof}
Note that fiber integration as well as the product of differential characters are natural operations~\cite[Def.~28 and~38]{BB}.
The first equality follows from commutativity of the diagram
\begin{equation}
\xymatrix{
M \ar[d]^{\ph}  & C^{\infty}(S,M)\times S \ar[d]^{A(\ph) \times \id}\ar[r]^-{\pr_2} \ar[l]_-{\ev}
& S \ar[d]^{\id}\\
M  &C^{\infty}(S,M)\times S \ar[r]^-{\pr_2} \ar[l]_-{\ev}& S\,.\\
}
\end{equation}
The second equality is derived from a similar commutative diagram
\begin{equation}
\xymatrix{
M \ar[d]^{\id} & C^{\infty}(S,M)\times S \ar[d]^{B(\psi) \times \psi}\ar[r]^-{\pr_2} \ar[l]_-{\ev}
& S \ar[d]^{\psi}\\
M &C^{\infty}(S,M)\times S \ar[r]^-{\pr_2} \ar[l]_-{\ev} & S\,,\\
}
\end{equation}
together with the invariance of fiber integration under orientation-preserving diffeomorphisms.
\end{proof}

Assume the degrees satisfy the relation \( \ell + r = \dim S \).
Then, by \cref{thm:H1isbundles}, the resulting differential character \( h \hatProduct g \in \widehat H^1(C^\infty(S, M),\TT) \) can be represented by a circle bundle \( \P \) over $C^{\infty}(S,M)$ with connection \( \nabla \).
From~\eqref{commuteswithoperations}, one sees that the curvature of $\nabla$ is 
$\alpha \hatProduct \beta \in \Omega^2_{\Z}(C^{\infty}(S,M))$, where $\alpha = \curv(h)$ and 
$\beta = \curv(g)$.
By \cref{prop:transgress}, the action $A(\phi)$ of 
$\phi \in \Diff(M,h)$ on 
$C^{\infty}(S,M)$ preserves the differential character $h \hatProduct g$ 
and thus the holonomy 
of $(\P,\nabla)$. The same holds for the action $B(\psi)$
of $\psi \in \Diff_+(S, g)$. 
By \cref{thm:H1isbundles}, the pull-back bundles
$A(\phi)^*(\P,\nabla)$ and $B(\psi)^*(\P,\nabla)$ are 
isomorphic to $(\P,\nabla)$, 
so there exist connection-preserving bundle automorphisms 
$\widehat{A}(\phi) \colon \P \rightarrow \P$ and
$\widehat{B}(\psi) \colon \P \rightarrow \P$
that cover $A(\phi)$ and $B(\psi)$.
The following theorem summarizes the discussion so far.
\begin{theorem}
	\label{prop::prequantumBundleOnFunctionSpace}
	Let $M$ and $S$ be finite dimensional manifolds, with $S$ compact and oriented.
	Let \( h \in \widehat H^\ell(M,\TT) \) and \( g \in \widehat H^r(S,\TT) \) be differential characters with curvature forms $\alpha \in \Omega^{\ell+1}_{\Z}(M)$ and $\beta \in \Omega^{r+1}_{\Z}(S)$, respectively.
	Assume that \( \ell + r = \dim S \).
	Then the differential character \( {h \hatProduct g} \in \widehat H^1(C^\infty(S, M),\TT) \) is the holonomy of a circle bundle \( \P \) over $C^{\infty}(S,M)$ with connection \( \nabla \)
	and curvature $\alpha \hatProduct \beta \in \Omega^2_{\Z}(C^{\infty}(S,M))$.
	Moreover, for every \( \phi \in \Diff(M,h) \) and \( \psi \in \Diff_+(S, g) \), there exist bundle automorphisms \( \widehat{A}(\phi)\) 
	and \(\widehat{B}(\psi)\) in \( \AutGroup(\P, \nabla) \) that cover 
	\( A(\phi), B(\psi) \in \Diff(C^\infty(S, M)) \).
\end{theorem}

\section{Central extensions of \texorpdfstring{$\Diff_{\ex}(M,\alpha)$
and $\Diff_{\ex}(S,\beta)$}{subgroups of the group of exact diffeomorphisms}}

We continue in the setting of \cref{prop::prequantumBundleOnFunctionSpace}, and 
use the bundle $(\cP,\nabla)$ 
to construct central extensions of $\Diff_{\ex}(M,\alpha)$
and $\Diff_{\ex}(S,\beta)$.
For $\Phi \in C^{\infty}(S,M)$, consider the restriction $\P_{\Phi} \rightarrow C^{\infty}(S,M)_{\Phi}$ of 
$\P \rightarrow C^{\infty}(S,M)$ to the connected component
of $\Phi$.
Denote by $\Diff(M,h,[\Phi]) \subseteq \Diff(M,h)$ and
$\Diff(S, g, [\Phi]) \subseteq \Diff_+(S, g)$ the subgroups that 
fix the homotopy class $[\Phi] \in \pi_{0}(C^{\infty}(S,M))$.
By continuity, these contain 
$\Diff_{\ex}(M,\alpha) \subseteq \Diff(M,h)$
and 
$\Diff_{\ex}(S, \beta) \subseteq \Diff(S, g)$.
The pull-back 
of the group extension 
\[
\TT \rightarrow \mathrm{Aut}(\P_{\Phi},\nabla) \rightarrow \Diff(C^{\infty}(S,M)_{\Phi})
\]
along the action 
$A \colon \Diff(M,h,[\Phi]) \rightarrow \Diff(C^{\infty}(S,M)_{\Phi})$
yields a 
central extension 
\begin{equation}\label{eq:extensionL}
\TT \rightarrow \widehat \Diff(M,h,[\Phi])\rightarrow \Diff(M,h,[\Phi])\,.
\end{equation}
Similarly, the pull-back along 
$B \colon \Diff(S,g,[\Phi]) \rightarrow \Diff(C^{\infty}(S,M)_{\Phi})$
yields the 
central extension 
\begin{equation}\label{eq:extensionR}
\TT \rightarrow \widehat \Diff(S,g,[\Phi]) \rightarrow \Diff(S,g, [\Phi])\,.
\end{equation}


\subsection{Smooth Lie group extensions}
We would like to say that these are \emph{smooth} extensions of Lie groups.
But, unfortunately, 
we are not aware of any Fr\'echet--Lie group structure on 
$\Diff(M,h,[\Phi])$ and $\Diff(S,g, [\Phi])$ except when
\( \alpha \) and \( \beta \) are either symplectic or volume forms and $M$ is compact.

In general, we therefore formulate the smoothness requirement as follows.
Consider a smooth action of a Fr\'echet--Lie group \( G \) on \( M \) that 
preserves \( h \) and \( [\Phi] \); similarly, let $H$ be a Fr\'echet--Lie group which acts smoothly on \( S \) and preserves \( g \) as well as \( [\Phi] \).
We denote the corresponding group homomorphisms by $i_{G} \colon G \to \Diff(M,h,[\Phi])$
and $i_{H} \colon H \to \Diff(S,g,[\Phi])$.
By \cref{gene}, the pull-back along $i_{G}$ and $i_{H}$
of the 
central extensions~\eqref{eq:extensionL} and~\eqref{eq:extensionR}
%
then yields \emph{smooth} Lie group extensions
\begin{equation}
	\label{eq:extensionsLiePrelim}
	\T \rightarrow \widehat{G} \rightarrow G
	\quad
	\text{and}
	\quad
	\T \rightarrow \widehat{H} \rightarrow H, 
\end{equation}
respectively.


Let $a \colon \X_{\ex}(M,\alpha) \to T_{\Phi}C^{\infty}(S,M)$ and $b \colon \X_{\ex}(S,\beta) \to T_{\Phi}C^{\infty}(S,M)$ be the infinitesimal versions at $\Phi$ of the actions \( A \) and \( B \), respectively.
That is, we have
\begin{equation}
	\label{eq:actionsInf}
	a(X) = X \circ \Phi, \qquad b(v) = -\Phi_*v .
\end{equation}
From \cref{prop::prequantumBundleOnFunctionSpace} and equation~\eqref{commuteswithoperations}
(\cf also~\cite[Eq.~2]{Vizman2}), one finds that the curvature of $\nabla$ at $\Phi \in C^{\infty}(S,M)$
is $(\alpha \hatProduct \beta)_{\Phi} \in \Alt^2 (T_{\Phi}C^{\infty}(S,M))$, with
\begin{equation}
	(\alpha \hatProduct \beta)_{\Phi}(U,V) = \int_{S}\Phi^* \bigl(i_{V}i_{U}(\alpha \circ \Phi)\bigr)\wedge \beta \, .
\end{equation}
Here, \( \Phi^* \bigl(i_{U}i_{V}(\alpha \circ \Phi)\bigr) \) is the differential \( (r-1) \)-form on $S$ assigning to $v_1, \dotsc, v_{r-1} \in T_s S$ the value 
\[ \alpha_{\Phi(s)}
\bigl(U(s), V(s), \Phi_* v_1, \dotsc, \Phi_* v_{r-1}\bigr).\]
The pull-back along $a$ and $b$ of the 
curvature $(\alpha \hatProduct \beta)_{\Phi}$ 
yields the Lie algebra 2-cocycles $\tau$ on $\X_{\ex}(M,\alpha)$ and 
$\nu$ on $\X_{\ex}(S,\beta)$ given by 
 \begin{equation}
\tau(X,Y) = \int_{S} (\Phi^* i_{Y}i_{X}\alpha) \wedge \beta\,,
\qquad
\nu(u, v) = \int_S (i_v i_u \Phi^* \alpha) \wedge \beta\,.
\end{equation}
For the former identity, we use that, for $U = a(X)$ and $V = a(Y)$, \cref{eq:actionsInf} yields
$\Phi^* \bigl(i_V i_U (\alpha \circ \Phi)\bigr) = \Phi^* i_{Y}i_{X}\alpha$. 
For the latter, note that the
fundamental vector fields $U = b(u)$ and \( V = b(v)\) give $\Phi^*\bigl(i_{V}i_{U}(\alpha \circ \Phi)\bigr) = i_{v} i_u \Phi^*\alpha$ according to \cref{eq:actionsInf}.
The Lie algebra cocycles on $\fg = \mathrm{Lie}(G)$ and $\fh = \mathrm{Lie}(H)$ corresponding to Lie group extensions~\eqref{eq:extensionsLiePrelim} are given by the pull-back along the Lie algebra actions $i_{\fg} \colon \fg \to \X_{\ex}(M,\alpha)$ and $i_{\fh} \colon \fh \to \X_{\ex}(S,\beta)$ of the cocycles $\tau$ and $\nu$.
In summary, we obtain the following.
\begin{theorem}
\label{prop::centralExtensionOfDiffeoGroups}
Let $M$ and $S$ be finite dimensional manifolds, and
let $h \in \widehat{H}^{\ell}(M,\TT)$ and 
$g \in \widehat{H}^r(S, \TT)$ be differential characters 
with curvature $\alpha \in \Omega_{\Z}^{\ell + 1}(M)$ 
and $\beta \in \Omega^{r+1}_{\Z}(S)$, respectively.
Assume that $S$ is compact and oriented, and that $\ell + r = \mathrm{dim}(S)$.
Let $i_{G} \colon G \to \Diff(M,h,[\Phi])$ and 
$i_{H} \colon H \to \Diff(S,g,[\Phi])$
be Fr\'echet--Lie groups that act smoothly on $M$ and~$S$, 
in such a way that $h$, $g$ and $[\Phi] \in \pi_{0}(C^{\infty}(S,M))$ are preserved.
Then the pull-back to $G$ and $H$ of~\eqref{eq:extensionL} and~\eqref{eq:extensionR} yields smooth central extensions 
\begin{equation}
\T \rightarrow \widehat{G} \rightarrow G\,, \qquad 
\T \rightarrow \widehat{H} \rightarrow H
\end{equation}
of Fr\'echet--Lie groups.
The corresponding Lie algebra $2$-cocycles 
are given by the pull-back along the infinitesimal action 
$i_{\fg} \colon \fg \to \X_{\ex}(M,\alpha)$ and 
$i_{\fh} \colon \fh \to \X_{\ex}(S,\beta)$ of the 
continuous Lie algebra 
2-cocycles
\begin{align}\label{eq:LAcocexM}
\tau \colon {\textstyle \bigwedge^2}\X_{\ex}(M,\alpha)\rightarrow \R, \quad   &\tau(X,Y) = \int_{S} (\Phi^*i_{Y}i_{X}\alpha) \wedge \beta\,,\\
\nu \colon {\textstyle \bigwedge^2}\X_{\ex}(S,\beta)\rightarrow \R, \quad   &\nu(u,v) = \int_{S} (i_{v}i_{u}\Phi^*\alpha) \wedge \beta = 
\int_{S} \Phi^*\alpha \wedge i_{v}i_{u}\beta
\,.\label{eq:LAcocexS}
\end{align}
\qedhere
\end{theorem}
Note that although the Lie algebra cocycles $\tau$ and $\nu$ (and, hence, the corresponding Lie algebra extensions) depend only on 
the forms $\alpha$ and $\beta$, the group extensions will in general depend 
on the choice of integrating differential characters $h$ and $g$.

\subsection{Extensions at the Lie algebra level}\label{subsec:ExtLAlevel}

\Cref{prop::centralExtensionOfDiffeoGroups} yields cocycles \( \tau \) and \( \nu \) on the Lie algebras \( \X_{\ex}(M,\alpha) \) and \( \X_{\ex}(S,\beta) \), respectively, 
that give rise to Lie group extensions.
We will now show that the Lie algebra extensions 
obtained in this way all factor through a common type of central extension.

Let $\Omega^{k}_{\mathrm{ex}}(M)$ 
denote the space of exact $k$-forms, and let
\[
\overline{\Omega}{}^{k}(M) \defeq \Omega^{k}(M)/\Omega^{k}_{\ex}(M)\,.
\]
For every closed $(k+1)$-form $\omega$, we define
\begin{equation}
	\hat\X_{\ex}(M,\omega)
		\defeq \set*{\bigl(v,[\psi_v]\bigr) \in \X_{\ex}(M,\omega) \times \overline{\Omega}{}^{k-1}(M)\given i_{v}\omega = \dif \psi_v }
\end{equation}
and denote the projection on the first factor by $\pi$.
If \( \omega \) is non-degenerate, then \( \hat\X_{\ex}(M,\omega) \) is canonically isomorphic to \( \overline{\Omega}{}^{k}(M) \).

\begin{proposition}\label{prop:ceLA} 
Let $M$ be a finite dimensional manifold, and let $\omega$ be a closed $(k+1)$-form on $M$.
With the Lie bracket on $\hat\X_{\ex}(M,\omega)$ defined by  
\begin{equation}\label{eq:bracketext}
\bigl[\bigl(v,[\psi_v]\bigr),\bigl(w,[\psi_{w}]\bigr)\bigr] \defeq \bigl([v,w], [i_v i_w \omega]\bigr) = \bigl([v,w], [L_v \psi_w]\bigr) \, ,
\end{equation}
the projection \( \pi \) yields a central extension
\begin{equation}\label{eq:extensionLA}
0 \rightarrow H^{k-1}(M,\R) \longrightarrow \hat\X_{\ex}(M,\omega) \stackrel{\pi}{\longrightarrow} 
\X_{\ex}(M,\omega) \rightarrow 0\,,
\end{equation}
of Fr\'echet--Lie algebras.
\end{proposition}

\begin{proof} (\cf~\cite[Thm.~13]{Ne05}, \cite[Rem.~1.11]{NV10}) 
To see that~\eqref{eq:bracketext} defines an element of 
$\hat\X_{\ex}(M,\omega)$, 
note that $L_{v}\omega = 0$ and that $i_w\omega$ is closed, so that
\[i_{[v,w]}\omega = L_{v}i_w \omega - i_w L_v \omega =
\dif (i_v i_w \omega) + i_v \dif (i_w\omega) = \dif (i_v i_w \omega)\,.\]
For the second equality in~\eqref{eq:bracketext}, note that
$L_{v}\psi_{w} = i_v i_w \omega +  \dif (i_v \psi_w)$.

By the de Rham isomorphism, a $(k-1)$-form $\gamma$ is exact if and only if 
it integrates to zero on all closed cycles. It follows that $\Omega^{k-1}_{\ex}(M)$
is a closed subspace of the Fr\'echet space $\Omega^{k-1}(M)$, and so the quotient 
$\overline{\Omega}{}^{k-1}(M)$ is a Fr\'echet space by~\cite[Thm.~1.41]{Rud91}.
Since $\hat\X_{\ex}(M,\omega)$ is the kernel of the continuous linear map 
\[\X(M) \times \overline{\Omega}{}^{k-1}(M) \rightarrow \Omega^{k}(M), \quad (v,[\psi]) \mapsto i_{v}\omega - \dif \psi \,,\]
it is a closed subspace of a Fr\'echet space, hence it is Fr\'echet itself. 
The continuity of the bracket and projection follows from the explicit description.

It remains to show that the bracket is a Lie bracket.
It is manifestly skew-symmetric. For the Jacobi identity, 
it suffices to show that the form
\[ i_{u}i_{[v,w]}\omega + i_{v}i_{[w,u]}\omega + i_{w}i_{[u,v]}\omega \] 
is exact. 
This expression equals
\begin{gather*}
(i_v i_{[w,u]} - i_{w}i_{[v,u]} - i_{[v,w]}i_{u}) \, \omega 
= (i_{v}L_{w} - i_{w}L_{v} - i_{[v,w]}) \, \dif \psi_{u} \\
=(L_{v}L_{w} - L_{w}L_{v} - L_{[v,w]}) \, \psi_{u} -
\dif (i_{v}L_{w} - i_{w}L_{v} - i_{[v,w]}) \, \psi_{u}\,,
\end{gather*}
where we used the formul\ae{} $i_{[X,Y]} = L_{X}i_{Y} - i_{Y}L_{X}$, $i_{X}\omega = \dif \psi_{X}$, and $L_{X}\omega = 0$
in the first equality, and $L_{X} \dif = \dif L_{X}$ and $i_{X} \dif = L_{X} - \dif i_{X}$ in the second.
Since $L_{[v,w]} = L_{v}L_{w}- L_{w}L_{v}$, we find
\[
i_{u}i_{[v,w]}\omega + i_{v}i_{[w,u]}\omega + i_{w}i_{[u,v]}\omega = 
\dif (i_{w}L_{v} - i_{v}L_{w} + i_{[v,w]}) \, \psi_{u}\,,
\]
so that it defines the zero class in $\overline{\Omega}{}^{k-1}(M)$ as required.
%
\end{proof}

Every continuous linear functional $\lambda \in \hat\X_{\ex}(M,\omega)^*$ 
gives rise to a central 
extension $ \X_{\ex}(M,\omega) \oplus_{\lambda} \R$, with Lie bracket
\begin{equation}
	\label{eq:centralExtensionLambda}
[v\oplus x, w \oplus y] = [v,w] \oplus \lambda\bigl(([v,w], [i_v i_w \omega])\bigr)\,.
\end{equation}
The isomorphism class of the extension obtained in this way depends only on the restriction 
of $\lambda$ to $H^{k-1}(M,\R)$, as 
the homomorphism 
$(\pi,\lambda) \colon \hat\X_{\ex}(M,\omega) \rightarrow \X_{\ex}(M,\omega) \oplus_{\lambda}\R$ 
drops to an isomorphism between
$\hat\X_{\ex}(M,\omega)/\ker(\lambda|_{H^{k-1}(M,\R)})$ and $\X_{\ex}(M,\omega) \oplus_{\lambda}\R$. 
This gives rise to a map 
\begin{equation}
 \Xi_{\omega} \colon H^{k-1}(M,\R)^* \rightarrow H^2(\X_{\ex}(M,\omega),\R).
\end{equation}
Summarizing, the restriction of $\lambda$ to $H^{k-1}(M,\R)$ determines the 1-dimensional central extension,
while $\lambda \in \hat\X_{\ex}(M,\omega)^*$ provides this extension with a splitting.

The central extension $\pi\colon \hat\X_{\ex}(M,\omega) \rightarrow \X_{\ex}(M,\omega)$
is in general not universal.
Its significance lies in the fact that it captures
the central extensions corresponding to the cocycles 
\eqref{eq:LAcocexM} and \eqref{eq:LAcocexS} from
\cref{prop::centralExtensionOfDiffeoGroups}.
For the cocycle $\tau$ on $\X_{\ex}(M,\alpha)$, 
note that  $\pi^*\tau = \difCE \sigma_{\beta}$,
where 
$\difCE$ denotes the differential of the Chevalley-Eilenberg complex, 
$\pi$ is the central extension $\pi\colon \hat\X_{\ex}(M,\alpha)
 \rightarrow \X_{\ex}(M,\alpha)$
of $\X_{\ex}(M,\alpha)$ by $H^{\ell-1}(M,\R)$, and 
$\sigma_{\beta}$ is the 1-cochain on $\hat\X_{\ex}(M,\alpha)$ given by 
\begin{equation}\label{eq:gramschaap}
\sigma_{\beta} \colon \hat\X_{\ex}(M,\alpha)
\rightarrow \R, \qquad \sigma_{\beta}\bigl((X,[\psi_{X}])\bigr) = \int_{S}(\Phi^*\psi_{X}) \wedge \beta \, . 
\end{equation}
Similarly, the cocycle $\nu$ of~\eqref{eq:LAcocexS} trivializes when it is pulled back along 
the extension $\pi\colon \hat\X_{\ex}(S,\beta) \rightarrow \X_{\ex}(S,\beta)$, as $\pi^*\nu = \difCE\bar{\sigma}_{\alpha}$
for the Lie algebra 1-cochain
\begin{equation}
\bar{\sigma}_{\alpha} \colon \hat\X_{\ex}(S,\beta) 
 \rightarrow \R, \qquad \bar{\sigma}_{\alpha} \bigl((u,[\psi_{u}])\bigr) = \int_{S}(\Phi^*\alpha)\wedge \psi_{u}\,. 
\end{equation}

\begin{proposition}\label{prop:twodiagrams}
The isomorphism class of the Lie algebra extension of 
$\X_{\ex}(M,\alpha)$ corresponding to the 2-cocycle $\tau$
is determined by the restriction 
of $\sigma_{\beta}$ to $H^{\ell-1}(M,\R) \subseteq \widehat{\X}_{\ex}(M,\alpha)$.
Similarly, the isomorphism class of the Lie algebra extension of $\X_{\ex}(S,\beta)$
corresponding to $\nu$ is determined by the restriction of 
$\bar{\sigma}_{\alpha}$ to $H^{r -1}(S,\R)$.
In other words,
the following diagrams commute.
\begin{equation*}
\xymatrix{
	\Omega_{\Z}^{r+1}(S)\ar[dr]_{\beta \mapsto \sigma_{\beta}}\ar[r]^{\beta \mapsto [\tau]\qquad} 
	&H^2(\X_{\mathrm{ex}}(M,\alpha),\R)
	&H^2(\X_{\mathrm{ex}}(S,\beta),\R)&\Omega^{\ell + 1}_{\Z}(M)
	\ar[l]_{\qquad\alpha \mapsto [\nu]}\ar[dl]^{\alpha \mapsto \bar{\sigma}_{\alpha}}
	\\
	& H^{\ell-1}(M,\R)^* \ar[u]_{\Xi_{\alpha}}&H^{r-1}(S,\R)^* \ar[u]^{\Xi_{\beta}} 
}
\end{equation*}
\end{proposition}

In general, not every central extension of Lie algebras integrates 
to a central extension of Lie groups, see \eg~\cite{Ne02}. 
Combining Proposition~\ref{prop:twodiagrams} 
with 
Theorem~\ref{prop::centralExtensionOfDiffeoGroups},
we identify a lattice of integrable classes.

\begin{theorem}[Compact version]\label{cor:rol}
Let $M$ be a compact orientable $n$-dimensional manifold, and let 
$h \in \widehat{H}^{k}(M,\TT)$ be a differential character with curvature $\omega \in \Omega^{k+1}(M)_{\Z}$. 
Let~$G$ be a Fr\'echet--Lie group, and let 
$i \colon G \rightarrow \Diff(M,h)$ be a smooth 
action on $M$ such that, for every $g\in G$, the diffeomorphism 
 $i(g) \colon M \rightarrow M$ preserves $h$ and is homotopic to the identity. 
Finally, let 
\[H^{k-1}(M,\R) \rightarrow \widehat{\fg} \rightarrow \fg\]
be the pull-back of the central extension \eqref{eq:extensionLA}
along the Lie algebra homomorphism $i_{\fg} \colon \fg \rightarrow \X_{\ex}(M,\omega)$.
Then the image of the integral classes $H^{n-k+1}(M,\R)_{\Z} \subseteq H^{k-1}(M,\R)^*$
under the map 
$i_{\fg}^*\,\circ\,\Xi_\om \colon H^{k-1}(M,\R)^* \rightarrow H^2(\fg,\R)$ 
consists 
of classes in $H^2(\fg,\R)$ that are integrable to central Lie group extensions of $G$ by~$\T$.
\end{theorem}

\begin{proof}
For $[\beta] \in H^{n-k+1}(M,\R)_{\Z}$, choose a differential character
$g\in \widehat{H}^{n-k}(M,\TT)$ 
with $\curv(g) = \beta$. 
By Theorem~\ref{prop::centralExtensionOfDiffeoGroups} with $M=S$, $\Phi = \mathrm{Id}$, 
and $\alpha = \omega$,
we obtain a central Lie group extension 
\begin{equation}
\label{eq:stensie}
\TT \rightarrow \widehat{G} \rightarrow G.
\end{equation}
(Indeed, the diffeomorphism $i(g)$ preserves $[\mathrm{Id}]\in \pi_0(C^{\infty}(M,M))$ if it is homotopic to the identity.)
The corresponding Lie algebra cocycle is the pull-back along 
$i_{\fg} \colon \fg \rightarrow \X_{\ex}(M,\omega)$
of the 2-cocycle $\tau(X,Y) = \int_{M}(i_{Y}i_{X}\omega)\wedge\beta $
on $\X_{\ex}(M,\omega)$ (cf. equation~\eqref{eq:LAcocexM}).
By Proposition~\ref{prop:twodiagrams}, the isomorphism class of 
the cocycle $\tau$ is determined by the restriction 
of $\sigma_{\beta}$ to $H^{k-1}(M,\R) \subseteq \widehat{\X}_{\ex}(M,\omega)$.
This is the pairing with $[\beta] \in H^{n-k+1}(M,\R)_{\Z}$.
In terms of the commutative diagram
\begin{equation*}
\xymatrix{
	0\ar[r]&H^{k-1}(M,\R)\ar@{=}[d]\ar[r] 
	&\hat\X_{\ex}(M, \omega)\ar[r]^{\pi}
	&\X_{\ex}(M,\omega)\ar[r] & 0
	\\
	0\ar[r]& H^{k-1}(M,\R)\ar[r]
	&\hat{\fg} \ar[r]^{\pi_\fg}\ar[u]^{\hat{i}_{\fg}} 
	&\fg \ar[r] \ar[u]^{i_{\fg}}
	& 0,
}
\end{equation*}
we have $\pi_{\fg}^*i_{\fg}^*\tau = \difCE(\widehat{i_{\fg}}^*\sigma_{\beta})$.
Since
$\widehat{i_{\fg}}$ is the identity on $H^{k-1}(M,\R)$, 
the restriction of $\pi_{\fg}^*i_{\fg}^*\tau$ to 
$H^{k-1}(M,\R)$ is the same as the restriction of $\sigma_{\beta}$, 
namely the pairing with $[\beta]$.
\end{proof}


In order to extend Theorem~\ref{cor:rol} to 
manifolds $M$ that are not necessarily compact or orientable, we use
`singular' functionals on $\widehat{\X}_{\ex}(M,\omega)$. Suppose
that $z \in H_{k-1}(M,\Z)$ can be represented as $z = \Phi_*[S]$, where
$[S]$ is the fundamental class of a closed, oriented manifold $S$
and $\Phi \colon S \rightarrow M$ is a smooth map.
Then, with $\beta = 1$, equation~\eqref{eq:gramschaap} 
yields the linear functional $\sigma(X,\psi_{X}) = \int_{S}\Phi^*\psi_{X}$ on $\widehat{\X}_{\ex}(M,\omega)$,
whose restriction to $H^k(M,\R)$ is the pairing with $\Phi_*[S]$.
This yields a way of building group extensions for those classes in $H_{k-1}(M,\Z)$ that can be represented by a closed, 
oriented, smooth manifold.

\begin{corollary}[Noncompact version]\label{cor:rol2}
Let $M$ be a manifold of dimension $n$, 
and let 
$h \in \smash{\widehat{H}^{k}}(M,\TT)$ be a differential character with curvature $\omega \in \Omega^{k+1}(M)_{\Z}$. 
Assume that a class in $H_{k-1}(M,\Z)$ can be realized as the push-forward 
$\Phi_{*}[S]$ of the fundamental class of a  
closed, oriented $(k-1)$-manifold $S$ along a smooth map $\Phi \colon S \rightarrow M$. 
Let~$G$ be a Fr\'echet--Lie group, and let 
$i \colon G \rightarrow \Diff(M,h)$ be a smooth 
action on $M$ such that, for all $g\in G$, 
$i(g) \colon M \rightarrow M$ preserves $h$, and $i(g)^*\Phi$ is homotopic to $\Phi$.
Then the Lie algebra extension of \( \fg \) corresponding to the image of $\Phi_*[S]$ under the map 
$i_{\fg}^*\,\circ\,\Xi _\om\colon H^{k-1}(M,\R)^* \rightarrow H^2(\fg,\R)$ 
is integrable to a central Lie group extension of $G$ by~$\T$.
\end{corollary}

In general, not every class $z \in H_{k-1}(M,\Z)$ can be represented as $z = \Phi_*[S]$.
By \cite[Thm.~9.1]{ConnerFloyd1964}, the classes with this property are \emph{Steenrod representable}; they constitute 
the image of the bordism group under the natural homomorphism 
$\Omega_{k-1}(M) \rightarrow H_{k-1}(M,\Z)$
that takes $[f\colon S \rightarrow M]$ to $f_*[S]$.
This image has been extensively investigated by Thom in \cite[especially \S 2.11]{Thom1954}.
If $H_{\bullet}(M,\Z)$ is finitely generated and has no odd torsion, then 
\emph{every} class is Steenrod representable by \cite[Thm.~15.2]{ConnerFloyd1964}.
In this case Corollary~\ref{cor:rol2} is simply a generalization of Theorem~\ref{cor:rol}
to noncompact manifolds.
Without any assumptions on $H_{\bullet}(M,\Z)$, one can still show that every class 
admits an odd multiple that is Steenrod representable \cite[Thm.~15.3]{ConnerFloyd1964}.
In particular, Corollary~\ref{cor:rol2} yields a lattice of full rank in $H^{k-1}(M,\R)^*$
that gives rise to integrable Lie algebra extensions, regardless whether $H_{\bullet}(M,\Z)$
has odd torsion or not.

If $M$ is compact and orientable, then the Lie algebra extensions of \cref{cor:rol2} are isomorphic to the ones in Theorem~\ref{cor:rol} by Poincar\'e duality, although the group extensions may well be different.


%

\subsection{The volume-preserving diffeomorphism group}
Let $\mu$ be an integral volume form on an $n$-dimensional compact manifold~$M$.
Then the extension~\eqref{eq:extensionLA} 
is the \emph{Lichnerowicz extension} (\cite{Lichnerowicz},~\cite[Sec.~10]{Ro95})
\begin{equation}\label{LichExt}
H^{n-2}(M,\R) \rightarrow \overline{\Omega}{}^{n-2}(M) \rightarrow \X_{\ex}(M,\mu).
\end{equation} 
By \cref{cor:rol},
the elements in the integral lattice $H_{n-2}(M,\Z) \subseteq H^{n-2}(M,\R)^*$ 
give rise to Lie group extensions of $\Diff_{\ex}(M,\mu)$.

In the remainder of this section, we illustrate how 
different geometric realizations of these integral 
elements of $H^{n-2}(M,\R)^*$ give rise to different representatives 
of the same integrable classes in $H^2(\X_{\ex}(M,\mu),\R)$.

\subsubsection{Singular Lichnerowicz cocycles}

Let \( h \) be a differential character with curvature \(\mu \). 
Choose a closed manifold \( S \) of dimension \(  n - 2 \), 
and apply \cref{prop::centralExtensionOfDiffeoGroups}
with $\alpha = \mu \in \Omega^{n}(M)_{\Z}$ and 
$\beta = k\in \Omega^0(S)_{\Z}$.
This
yields a smooth central extension
\begin{equation}
	\TT \to \widehat{\Diff}_{\ex}(M, \mu) \to \Diff_{\ex}(M,\mu)
\end{equation}
of the exact volume-preserving diffeomorphism group, 
integrating the \emph{singular Lichnerowicz cocycle} on
 \( \X_{\ex}(M, \mu) \)
\begin{equation}
	\tau(X, Y) = k\int_{S} \Phi^*i_Y i_X \mu.
\end{equation}
These group extensions are closely related to 
\emph{Ismagilov's central extensions}~\cite[Sec.~25.3]{Ismagilov}
of the Lie group $\Diff_{\ex}(M,\mu)$.


\subsubsection{Lichnerowicz cocycles}
As for the Hamiltonian actions in \S\ref{sec:intro}, we can change 
the role of \( \alpha \) and \( \beta \) to obtain (possibly nonequivalent) group extensions corresponding 
to cohomologous Lie algebra cocycles.
If $\omega$ is a closed, integral 2-form on $M$, then we can apply
\cref{prop::centralExtensionOfDiffeoGroups} to $S = M$, 
with $\alpha = k \omega$, $\beta = \mu$, and
$\Phi \colon S \rightarrow M$ the identity map.
For every choice of differential characters $g$ and $h$ 
with $\curv(g) = \mu$ and $\curv(h) = k\omega$, we then obtain a central extension of $\Diff_{\ex}(M,\mu)$
that integrates the \emph{Lichnerowicz cocycle} on \( \X_{\ex}(M, \mu) \),
\begin{equation}
	\nu(X, Y) = k \int_M \omega (X, Y)\, \mu.
\end{equation}
If \( \Phi_*[S] \in H_{n-2}(M,\Z) \) is Poincar\'e dual to \( -[\omega] \in H^2(M,\R)_{\Z} \), then the singular Lichnerowicz cocycle \( \tau \) is cohomologous to \( \nu \)
\cite{Vizman,JV15}. 

\subsubsection{Other representations}
Let \( N \) be an \( (n-1) \)-dimensional manifold.
Consider a differential character \( F \) of degree \( 0 \) on \( N \) with curvature \( 1 \)-form \( \rho \).
The character \( F \) can be viewed as a \( \T \)-valued function \( F \colon N \to \T \) on \( N \) and, from this viewpoint, \( \rho \) is then the logarithmic derivative \( \delta F \) of \( F \).
Let \( \Psi \colon N \to M \) be a smooth map.
By \cref{prop::centralExtensionOfDiffeoGroups} (with \( S = N \), \( \alpha = \mu \), \( \beta = \delta F \), and \( \Phi = \Psi \)), there exists a smooth central extension of \( \Diff_{\ex}(M,\mu, [\Psi]) \) that integrates the following Lie algebra cocycle on \( \X_{\ex}(M, \mu) \):
\begin{equation}
	\kappa(X, Y) = \int_N \Psi^* (i_Y i_X \mu) \wedge \delta F .
\end{equation}
If the functional on \( H^{n-2}(M, \R) \) given by
\begin{equation}
	\equivClass{\alpha} \mapsto \int_N \Psi^* \alpha \wedge \delta F
\end{equation}
coincides with the one given by \( \Phi_*[S] \in H_{n-2}(M,\Z) \) or by \( -[\omega] \in H^2(M,\R)_{\Z} \), then the cocycle \( \kappa \) is cohomologous to the Lichnerowicz cocycles \( \tau \) and \( \nu \), respectively.
In the context of hydrodynamics, the pair \( (N, \beta) \) corresponds to a codimension one singular membrane in the ideal fluid on \( M \).
We refer to~\cite{GayBalmazVizman2019} for further details.

\subsubsection{Outlook}
On the Lie algebra level, the central $\R$-extensions of $\X_{\ex}(M)$  constructed from 
functionals $\lambda \in H^{n-2}(M,\R)^*$ all
descend from a \emph{single} central extension of $\X_{\ex}(M)$ by $H^{n-2}(M,\R)$,
namely the Lichnerowicz extension \eqref{LichExt}. 
It is natural to ask whether, at the group level, the central $\T$-extensions constructed above
descend from a single central Lie group extension of $\Diff_{\ex}(M)$ by $H^{n-2}(M,\T)$. 

If $M$ is a 3-dimensional compact, connected, orientable 
manifold, we believe that this is indeed the case. To construct the required extension, let $h \in \widehat{H}^2(M,\T)$ be a differential character whose curvature is the volume form $\mu$.
Since $h$ is of degree $2$, the principal circle bundle $\cP \rightarrow C^{\infty}(S^1,M)$ 
can be viewed as the bundle obtained from transgression of the $U(1)$-bundle gerbe corresponding to $h$ (cf.~\cite{Brylinski2007}).
In fact, since the circle bundle $\cP \rightarrow C^{\infty}(S^1,M)$ is obtained by transgression of a bundle gerbe,
it is not only endowed with a connection, but also with a \emph{fusion structure}~\cite{Wa16}. It is therefore natural to consider 
the group of automorphisms of $\cP \rightarrow C^{\infty}(S^1,M)$ that preserve both the connection and -- 
in an appropriate sense -- also the fusion structure. Since the group of \emph{vertical} automorphisms of 
this type is isomorphic to $H^1(M,\T)$, 
this is a natural candidate for a central Lie group extension. We are currently pursuing this line of reasoning in 
joint work with Peter Kristel.

\appendix

\section{Central extensions of non-connected groups} 
\label{app:a.1}

In this appendix we show how to extend \cite[Thm.~3.4]{NV} 
concerning the Lie group structure on the central extension of a connected Lie group $G$, obtained by pull-back
of the prequantization central extension,
to non-connected Lie groups.
The results also work for locally convex manifolds. 

Let $Z$ be a connected abelian Lie group 
of the form $\fz/\Gamma$,
where $\fz$ is a Mackey complete vector space and  
$\Gamma \subeq \fz$ is a discrete subgroup. 
Furthermore, let 
$q : \P \to \M$ be a $Z$-bundle over the connected locally convex manifold $\M$, 
and let $\theta$ be a connection $1$-form with curvature $\omega$. 
We further assume that 
$\si : G \to \Diff(\M,\omega)$ 
defines a smooth action of the Lie group $G$ on $\M$ whose image
lies in the group $\Diff_{\hol}(\M)$ of holonomy-preserving diffeomorphism. 
According to~\cite{NV}, this implies that $\si(G)$ lies in the image of 
$\Aut(\P,\th)$, so that the pull-back
\[ \hat G \defeq \si^*\Aut(\P,\th) = G \times_{\Diff_{\hol}(\M)} \Aut(\P,\th) \] 
of the extension \( \Aut(\P, \theta) \to \Diff_{\hol}(\M) \) is a central $Z$-extension of $G$. 
The following theorem is a generalization of~\cite[Thm.~3.4]{NV} to non-connected
 Lie groups. 

\begin{theorem}\label{gene} 
If $\cM$ is connected and $G$ not necessarily connected with 
$\si(G) \subset\Diff_{\hol}(\M)$, then $\hat G$ carries a Lie group structure 
for which it is a central Lie group extension by $Z$.
Moreover, the action of $\hat G$ on $\P$ defined  by 
$(g, \ph) \cdot p = \ph(p)$ for \( g \in G \) and \( \ph \in \Aut(\P, \theta) \) is smooth. 
\end{theorem}

\begin{proof}
Pick $m_0 \in \cM$ and let $\sigma^{m_0} : G \to \M, g \mapsto \si(g)m_0$ be the corresponding 
smooth orbit map. We endow $\hat G$ with the smooth manifold structure obtained by identifying it 
with the $Z$-bundle $(\sigma^{m_0})^*\P = G \times_\M \P$ 
over $G$;~\cite[Lemma~3.2]{NV} also works for non-connected groups~$G$. 
In~\cite[Thm.~3.4]{NV} we have already seen that we thus obtain a Lie group 
structure on the central extension $\hat G_0$ of the identity component $G_0$ of $G$. 
We now show that the smooth structure extends to all of $\hat G$. 

The elements of $\hat G$ are of the form $\hat g = (g, \hat{\si_g})$, where the quantomorphism $\hat\si_g$ projects to $\si_g$.
Since every $\hat g \in \hat G$ acts as an automorphism on $\P$, the left multiplication 
by $\hat g = (g, \hat{\si_g})$ in 
$\hat G \cong (\sigma^{m_0})^*\P \subset G\times \P$ is smooth. 
It therefore remains to show that the 
inner automorphisms of $\hat G$ are smooth. 
Fix $\hat g_0 = (g_0, \hat{\si_{g_0}})$ in $\hat G$. For 
$\hat g=(g, \hat{\si_g}) \in \hat G$ we then have 
\[ \hat g_0^{} \hat g \, \hat g_0^{-1} 
= (g_0^{} g g_0^{-1}, \hat{\si_{g_0}}\hat{\si_{g}}\hat{\si_{g_0}}^{-1}).\] 
Pick a point $y_0 \in \P_{m_0}$. We have to show that the map 
\[ \hat G \to \P, \quad (g, \hat{\si_g}) \mapsto 
\hat{\si_{g_0}}\hat{\si_{g}}\hat{\si_{g_0}}^{-1}(y_0) \] 
is smooth. This will follow if all orbits maps for $\hat G$ on $\P$ are smooth, 
but this is ensured by the smoothness of the action of the connected Lie group $\hat G_0$ on $\P$ 
(\cite{NV}). 
\end{proof}

\section{Circle bundles and differential characters}
\label{sec::diffCharactersDegreeOne}
In this appendix we give a direct proof of the well-known correspondence between equivalence classes of principal circle bundles with connection and differential characters of degree \( 1 \).
Alternatively, this can be derived from the fact that principal circle bundles with connection are classified by Deligne cohomology~\cite[Thm.~2.2.12]{Brylinski2007}, which is a different model for differential cohomology, \cf~\cite[Sec.~5.2]{BB}.
\HOneIsBundles*
\begin{proof}
If two principal $\TT$-bundles with connection $P$ and $Q$ over $M$ have the same holonomy 
$h \colon Z_{1}(M)\rightarrow \TT$, then we show that they are isomorphic 
by trivializing
$\overline{P} \otimes_{\TT}Q$,
where \( \overline{P} \) has the same underlying manifold as \( P \) but inverse \( \TT \)-action.
A flat section 
$s \colon M \rightarrow \overline{P} \otimes_{\TT}Q$ is obtained as follows.
Choose a base point $m_i$ in every connected component $M_i$ of $M$,  
and choose $p_i \in \overline{P} \otimes_{\TT}Q$ that project to $m_i$.
(The section will depend on these choices in a controlled way.)
For a point $m\in M$, choose a piecewise smooth path $\gamma \colon [0,1] \rightarrow M$
from $\gamma(0) = m_i$ to $\gamma(1) = m$, and let $\overline{\gamma}$ be its horizontal lift to
$\overline{P} \otimes_{\TT}Q$ with initial point $\overline{\gamma}(0) = p_i$.
Then $s(m) \defeq \overline{\gamma}(1)$.   
The fact that the holonomies coincide guarantees that 
$s(m)$ does not depend on the choice of path, and $s$ is smooth
because in a local trivialization it takes the form $s(x) = (x, \exp(\int_{\straightpath{0x}} A))$, 
where $\straightpath{0x}$ is a straight line from $0$ to $x$ and $A$ is the 1-form that represents
the connection in the given trivialization.

It remains to show that for every $h \in \widehat{H}^1(M,\TT)$, we can construct a principal $\TT$-bundle 
$P \rightarrow M$ whose connection  represents $h$.
Consider the 
groupoid $PM \rightrightarrows M$ of piecewise smooth paths, 
and let $\mathcal{G} \rightrightarrows M$ be the groupoid 
\[\mathcal{G} \defeq (PM \times \TT) /\sim\,,
\text{ where }
(\ga_1,z_1) \sim (\ga_2,z_2)
\,\Leftrightarrow\,
h(\ga_{1} \ast \ga_2^{-1}) = z_1^{-1} z_2\,.
\]
To see that \( \mathcal{G} \) is a Lie groupoid, we define charts as follows:
for \( x, y \in M \), choose open convex coordinate neighborhoods \( U_x, U_y \subseteq M \) centered at these points.
The chart is labeled by a path \( \gamma \) from \( x \) to \( y \)
\begin{equation}
	\Phi_\gamma: U_x \times U_y \times \TT \to \mathcal{G}, \quad (u,v,z) \mapsto \equivClass{\straightpath{ux} \ast \gamma \ast \straightpath{yv}, z},
\end{equation}
where \( \straightpath{ux} \) denotes the straight path from \( u \in U_x \) to \( x \) in the coordinate neighborhood \( U_x \).
To verify smoothness of the chart transition maps, first consider the case where \( \gamma \) and \( \gamma' \) both start at \( x \) and end at \( y \).
Then the chart transition \( T_{\gamma \gamma'} = \Phi_{\gamma'}^{-1} \circ \Phi_\gamma \) takes the form
\begin{equation}\begin{split}
	T_{\gamma \gamma'}: U_x \times U_y \times \TT &\to U_x \times U_y \times \TT, \\
	(u,v,z) &\mapsto (u,v, h(\straightpath{ux} \ast \gamma \ast \straightpath{yv} \ast \straightpath{vy} \ast \gamma'^{-1}\ast \straightpath{xu}) \cdot z).
\end{split}\end{equation}
Since \( h \) is invariant under thin homotopies, this expression simplifies to \( T_{\gamma \gamma'}(u,v,z) = (u,v, h(\gamma \ast \gamma'^{-1})\cdot z) \), which is smooth.

Secondly, consider the case where \( \gamma \) goes from \( x \) to \( y \) and \( \gamma' \) from \( x' \) to \( y' \). By the above argument, we can choose the path \( \gamma' \) to be \( \gamma' = \straightpath{x' u_0} \ast \straightpath{u_0 x} \ast \gamma \ast \straightpath{yv_0} \ast \straightpath{v_0 y'} \), where \( u_0 \) and \( v_0 \) are fixed but otherwise arbitrary points in \( U_x \cap U_{x'} \) and \( U_y \cap U_{y'} \), respectively.
A similar calculation as above shows that \( T_{\gamma\gamma'}(u,v,z) = (u,v,h(\eta_{u,v})\cdot z) \), where
\begin{equation}
 	\eta_{u,v} \defeq \textcolor{red}{\straightpath{ux} \ast \gamma \ast \straightpath{yv}} \ast \textcolor{blue}{\straightpath{vy'} \ast \straightpath{y'v_0} \ast \straightpath{v_0y} \ast \gamma^{-1} \ast \straightpath{xu_0} \ast \straightpath{u_0x'} \ast \straightpath{x'u}}. 	
\end{equation}
 Up to thin homotopy \( \eta_{u,v} \) is the boundary of the two diamonds with edges \( u_0, x, u, x' \) and \( v_0, y, v, y' \), and thus \( h(\eta_{u,v}) \) is given by integration of \( \curv(h) \) over the diamonds. Since this operation depends smoothly on \( u \) and \( v \), the transition map \( T_{\gamma\gamma'} \) is smooth.
 
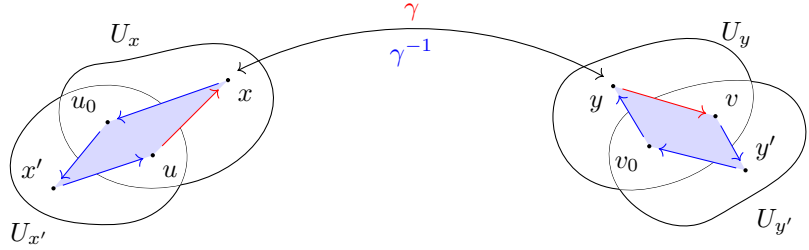
\begin{center}
\begin{tikzpicture}[scale=0.8]
	\tikzset{every label/.style={inner sep=0pt}}
	\clip (0.5,-0.85) rectangle (14.25, 3.55);

	\draw[draw=black,use Hobby shortcut,closed=true]
		(0.9,0.3) to [curve through ={(1.2,1.2) (1.6,1.55) (3.8,0.8) (3.1,-0.4) (2.3, -0.5) (1.4,-0.3)}] (0.9,0.3);
	\draw[draw=black,use Hobby shortcut,closed=true]
		(1.7,1.3) to [curve through ={(2.1,2.1) (2.6,2.2) (5.1,2.1) (3.9,0.2) (2.4,0.3)}] (1.7,1.3);
	\begin{scope}
		\clip[use Hobby shortcut,closed=true]
			(0.9,0.3) to [curve through ={(1.2,1.2) (1.6,1.55) (3.8,0.8) (3.1,-0.4) (2.3, -0.5) (1.4,-0.3)}] (0.9,0.3);
		\fill[blue!0,use Hobby shortcut,closed=true]
			(1.7,1.3) to [curve through ={(2.1,2.1) (2.6,2.2) (5.1,2.1) (3.9,0.2) (2.4,0.3)}] (1.7,1.3);
	\end{scope}

	\node[label=above left:\( x' \)] at (1.6,0.1) (xprime) {};
	\node[label=below right:\( u \)] at (3.25,0.65) (u) {};
	\node[label=above left:\( u_0 \)] at (2.5,1.2) (u0) {};
	\node[label=below right:\( x \)] at (4.5,1.9) (x) {};

	\fill[blue!15] (x.center) -- (u.center) -- (xprime.center) -- (u0.center) -- cycle;

	\fill (xprime) circle (1pt);
	\fill (u) circle (1pt);
	\fill (u0) circle (1pt);
	\fill (x) circle (1pt);
	
	\node at (2.8, 2.65) {\( U_x \)};
	\node at (1.2, -0.65) {\( U_{x'} \)};

	\draw[draw=black,use Hobby shortcut,closed=true]
		(9.9,1.3) to [curve through ={(10.2,2) (10.6,2.2) (13.1,2.1) (11.9,0.2) (10.4,0.3)}] (9.9,1.3);
	\draw[draw=black,use Hobby shortcut,closed=true]
		(10.9,0.3) to [curve through ={(11.2,1.4) (11.6,1.6) (14.1,1.1) (12.9,-0.2) (12.3, -0.5) (11.4,-0.3)}] (10.9,0.3);
	\begin{scope}
		\clip[use Hobby shortcut,closed=true]
			(9.9,1.3) to [curve through ={(10.2,2) (10.6,2.2) (13.1,2.1) (11.9,0.2) (10.4,0.3)}] (9.9,1.3);
		\fill[blue!0,use Hobby shortcut,closed=true]
			(10.9,0.3) to [curve through ={(11.2,1.4) (11.6,1.6) (14.1,1.1) (12.9,-0.2) (12.3, -0.5) (11.4,-0.3)}] (10.9,0.3);
	\end{scope}
	
	\node[label=above right:\( y' \)] at (13.1,0.4) (yprime) {};
	\node[label=above right:\( v \)] at (12.6,1.3) (v) {};
	\node[label=below left:\( v_0 \)] at (11.5,0.8) (v0) {};
	\node[label=below left:\( y \)] at (10.9,1.8) (y) {};

	\fill[blue!15] (y.center) -- (v.center) -- (yprime.center) -- (v0.center) -- cycle;

	\fill (yprime) circle (1pt);
	\fill (v) circle (1pt);
	\fill (v0) circle (1pt);
	\fill (y) circle (1pt);	
	
	\node at (12.95, 2.65) {\( U_y \)};
	\node at (13.6, -0.45) {\( U_{y'} \)};

	\draw[->, red] (u) -- (x);
	\draw[<->, black, use Hobby shortcut]
		(x.north east) .. 
		node[above=0.2,pos=5] {\textcolor{red}{\( \gamma \)}}
		node[below=0.2,pos=5] {\textcolor{blue}{\( \gamma^{-1} \)}}
		(5.2,2.3) .. (y.north west);
	\draw[->, red] (y) -- (v);
	\draw[->, blue] (v) -- (yprime);
	\draw[->, blue] (yprime) -- (v0);
	\draw[->, blue] (v0) -- (y);
	\draw[->, blue] (x) -- (u0);
	\draw[->, blue] (u0) -- (xprime);
	\draw[->, blue] (xprime) -- (u);
\end{tikzpicture}
\end{center}


If one chooses a base point $m_i$ in every connected component of $M$,
then the union $P \defeq \bigsqcup_{m_i} s^{-1}(m_i)$ of the source fibers of $\mathcal{G}$
is a principal $\TT$-bundle over $M$.
The groupoid homomorphism $\F: PM \rightarrow \mathcal{G}, \F(\ga) = [\ga,1]$
satisfies $\F(\sigma) = [\straightpath{mm},h(\sigma)]$ if $\sigma$ starts and ends at the same point $m$.
For a smooth path $c \colon [0,1] \rightarrow M$ with $c(0) = m$
and $\frac{d}{dt}\big|_{t=0}c = v$, let $c_{t} \colon [0,1] \rightarrow M$ be defined by \( c_t(s) = c(ts) \).
Define the horizontal lift of $v$ as $ \frac{d}{dt}\big|_{t=0}\F(c_{t}) \cdot [\ga,z]= \frac{d}{dt}\big|_{t=0}[\ga \ast c_t,z]\in T_{[\ga,z]}P$. This connection has holonomy $h$.
\end{proof}


\bibliographystyle{new}

\begin{thebibliography}{CFRZ16}


\bibitem[BHR10]{BaezHoffnungRogers2010}
J. Baez, A. Hoffnung, and C. Rogers,
{\it Categorified symplectic geometry and the classical string},
Commun. Math. Phys. 293 (2010), 701--725

\bibitem[Ba78]{Banyaga1978}
A. Banyaga,
\emph{Sur la structure du groupe des diff\'eomorphismes qui pr\'eservent
une forme symplectique},
Comment. Math. Helv. 53:2 (1978), 174--227


\bibitem[BB14]{BB} C. B\"ar and C. Becker, ``Differential Characters,'' 
Lecture Notes in Mathematics 2112, Springer, Cham, 2014 


\bibitem[Br07]{Brylinski2007} J.-L. Brylinski,
``Loop Spaces, Characteristic Classes and Geometric Quantization,'' Birkh\"auser Boston, 2007

\bibitem[Ca70]{Calabi1970}
E. Calabi,
\emph{On the group of automorphisms of a symplectic manifold},
in `Problems in analysis' ({L}ectures at the {S}ymposium in honor of
{S}alomon {B}ochner, {P}rinceton {U}niv., {P}rinceton,
{N}.{J}., 1969), pp. 1--26, Princeton Univ. Press, 1970

\bibitem[CFRZ16]{CalliesFregierRogersZambon2016}
M. Callies, Y. Fr\'egier, C. Rogers and M. Zambon,
\emph{Homotopy moment maps},
Adv. Math. 303 (2016), 954--1043

\bibitem[CF64]{ConnerFloyd1964}
P.E. Connor and E. E. Floyd,
``Differentiable Periodic Maps'', Ergebnisse der Mathematik und ihrer Grenzgebiete, Band 33, 
Springer Berlin Heidelberg, 1964



\bibitem[CS85]{CS85} J. C. Cheeger, J. S. Simons, \emph{Differential characters and 
geometric invariants}, in ``Geometry and Topology'' (College Park, Md., 1983/84), Lecture Notes in Math. 1167, 
Springer, Berlin 1985, 50--80

\bibitem[DJNV]{DJNV}
T. Diez, B.\ Janssens, K.-H. Neeb and C.\ Vizman,
{\it Central extensions of current groups}, 
to appear

\bibitem[FRS14]{FRS14}
D. Fiorenza, C. Rogers and U. Schreiber
\emph{$L_{\infty}$-algebras of local observables from higher prequantum bundles},
Homol. Homotopy Appl. 16 (2014), 107--142


\bibitem[GBV19]{GayBalmazVizman2019}
F. Gay-Balmaz and C. Vizman, 
{\it Vortex membranes in ideal fluids, coadjoint orbits, and characters}, 
arXiv:1909.12485 [math.SG]

\bibitem[Gl06]{Gloeckner} H.\ Gl\"ockner,
\emph{Implicit functions from topological vector spaces to Banach spaces},
Israel Journal Math. 155 (2006), pp.~205--252

\bibitem[HV04]{HV}
S. Haller and C. Vizman, {\it Non--linear Grassmannians as coadjoint orbits},
Math. Ann. 329:3 (2004), 771--785


\bibitem[Is96]{Ismagilov} R. S. Ismagilov, 
``Representations of infinite-dimen\-sional groups,''
Translations of Mathematical Monographs 152,
American Math. Soc., Providence, RI, 1996 

\bibitem[JV15]{JV15}
B.\ Janssens and C.\ Vizman,
{\it Central extensions of Lie algebras of symplectic and divergence free vector fields}, 
Banach Center Publ. 110, \emph{Geometry of jets and fields}, 105--114, 
Polish Acad. Sci. Inst. Math., 2016

\bibitem[Ko70]{Kostant} B. Kostant, {\it Quantization and unitary representations},
in ``Lectures in Modern Analysis and Applications III,''
Lecture Notes in Math. 170, 1970, Springer, Berlin

\bibitem[KM97]{KrMi97}
A. Kriegl  and P.~W. Michor,
``The Convenient Setting of Global Analysis,''
Mathematical Surveys and Monographs 53,
American Mathematical Society, Providence, RI, 1997 

\bibitem[Li74]{Lichnerowicz}
A. Lichnerowicz, 
{\it Alg\`ebre de Lie des automorphismes 
infinit\'esimaux d'une structure unimodulaire},
Ann. Inst. Fourier 24 (1974), 219--266

\bibitem[Lo94]{Losik1994}
M. V. Losik, 
{\it Categorical Differential Geometry}, 
Cah. Topol. G\'eom. Diff\'er. Cat\'eg., 35(4), 274--290, 1994


\bibitem[Ne02]{Ne02} K.-H. Neeb, 
{\it Central extensions of infinite-dimensional Lie groups}, 
Annales de l'Inst.\ Fourier 52 (2002), 1365--1442  

\bibitem[Ne05]{Ne05} ---,
{\it Lie algebra extensions and higher order cocycles},
J. Geom. Symmetry Phys. {\bf 4} (2005), 1--27

\bibitem[NW08]{NW08}
K.-H. Neeb and F.\ Wagemann,
{\it Lie group structures on groups of smooth and holomorphic maps on non-compact manifolds},
Geom. Dedicata 134
 (2008), 17--60

\bibitem[NV03]{NV}
K.-H. Neeb and C. Vizman, 
{\it Flux homomorphism and principal bundles over infinite dimensional manifolds},
Monatsh. Math. 139 (2003), 309--333 

\bibitem[NV10]{NV10} ---, {\it An abstract setting 
for hamiltonian actions}, Monatsh. Math. 159:3 (2010), 261--288



\bibitem[Ro95]{Ro95}
C.\ Roger, {\it Extensions centrales d'alg\`ebres et de groupes de Lie de dimension 
infinie,alg\`ebre de Virasoro et g\'en\'eralisations},
Rep. Math. Phys. 35 (1995), 225--266

\bibitem[R12]{Rogers2012}
C. Rogers,
{\it $L_{\infty}$-algebras from multisymplectic geometry},
Lett. Math. Phys. 100 (2012), 29--50


\bibitem[Rud91]{Rud91}
W.\ Rudin, ``Functional Analysis'', International Series in Pure and Applied Mathematics. 
McGraw--Hill inc., New York, 2nd ed., 1991

\bibitem[So70]{Souriau}
J.-M. Souriau, ``Structure des Syst\`emes Dynamiques -- Ma\^itrises de Math\'ematiques'', Dunod, 
1970 

\bibitem[Th54]{Thom1954}
R. Thom, {\it Quelques propri\'et\'es globales des vari\'et\'es diff\'erentiables},
Comment. Math. Helv. 28 (1954), 17--86

\bibitem[Vi10]{Vizman}
C. Vizman, {\it Lichnerowicz cocycles and central Lie group extensions}, 
Analele Universit\u a\c tii de Vest Timi\c soara,
Seria Matematic\u a--Informatic\u a XLVIII (2010), 285--297

\bibitem[Vi11]{Vizman2} 
---, {\it Induced differential forms on manifolds of functions}, 
Archivum Mathematicum 47 (2011), 201--215

\bibitem[Wa16]{Wa16}
Konrad Waldorf,
{\it Transgression to Loop spaces and its inverse, II: Gerbes and fusion bundles with connection},
Asian J. Math {\bf 20:1}, 59--116.

\bibitem[Za12]{Zambon2012}
M. Zambon,
{\it $L_{\infty}$-algebras and higher analogues of Dirac structures and Courant algebroids},
J. Symplect. Geom. 10(4) (2010), 563--599

\end{thebibliography}

\vspace*{7ex}

\noindent Tobias Diez. Max Planck Institute for Mathematics in the Sciences, 04103 Leipzig, Germany and Institut f\"ur Theoretische Physik, Universität Leipzig, 04009 Leipzig, Germany and Institute of Applied Mathematics, Delft University of Technology, 2628 XE Delft, The Netherlands. \texttt{tobias.diez@mis.mpg.de}
\\[1ex]
\noindent Bas Janssens. Institute of Applied Mathematics, Delft University of Technology, 2628 XE Delft, The Netherlands. \texttt{b.janssens@tudelft.nl}
\\[1ex]
\noindent Karl-Hermann Neeb. Department of Mathematics, FAU Erlangen-N\"urnberg, 91058 Erlangen, Germany. \texttt{neeb@math.fau.de}
\\[1ex]
\noindent Cornelia Vizman. Department of Mathematics, West University of Timi\c soara. RO--300223 Timi\c soara. Romania. \texttt{cornelia.vizman@e-uvt.ro}
\end{document}